\documentclass[10pt, oneside]{article}   	
\usepackage{amsmath} 
\usepackage{amssymb}  

\title{\LARGE \bf
Selective model-predictive control for flocking systems.
}

\usepackage{geometry}                

\usepackage{graphicx}
\usepackage{authblk}
\usepackage{amssymb}
\usepackage{caption}
\usepackage{epstopdf}
\usepackage{rotating}
\usepackage[sans]{dsfont}
\usepackage[T1]{fontenc}
\usepackage[english]{babel}
\usepackage{latexsym}
\usepackage{mathrsfs}
\usepackage{amscd}
\usepackage{sidecap}
\usepackage{graphicx}
\usepackage{color}
\usepackage{stmaryrd}
\usepackage{algorithmic}
\usepackage{algorithm}

\usepackage{amsmath}
\usepackage{amsfonts}
\numberwithin{equation}{section}
	\usepackage{enumerate}
\usepackage{amsthm}
\usepackage[numbers,sort]{natbib}

\usepackage{array}
\newcolumntype{P}[1]{>{\centering\arraybackslash}p{#1}}
\newcolumntype{M}[1]{>{\centering\arraybackslash}m{#1}}

\newcommand{\vv}{\mathbf{v}}

\newcommand{\Id}{\mathrm{Id}}

\newcommand{\be}{\begin{equation}}
\newcommand{\ee}{\end{equation}}

\newcommand{\epsi}{\varepsilon}
\newcommand{\ba}{\begin{array}}
\newcommand{\ea}{\end{array}}
\newcommand{\pM}{\left(\begin{array}}
\newcommand{\Mp}{\end{array}\right)}

\newcommand{\LL}{{\bf L} }
\newcommand{\eee}{{\bf e}}

\newcommand{\A}{{\bf A}}
\newcommand{\BB}{{\bf B}}

\newcommand{\QQ}{{\bf S}}

\newcommand{\Rddd}{\mathbb{R}^{2d}\times\mathbb{R}^{2d}}
\newcommand{\Rdd}{\mathbb{R}^d\times\mathbb{R}^d}

\newcommand{\Qs}{S}
\newcommand{\sidecap}[1]{ {\begin{sideways}\parbox{3.cm}{\centering#1}\end{sideways}} }
\def\u{u}
\def\v{{v}^{n+1}}
\def\p{{p}^{n+1}}

\graphicspath{{pics_arxiv/}}

\newtheorem{thm}{Theorem}[section]

\newtheorem{rmk}{Remark}[section]

\newtheorem{definition}{Definition}[section]
\newtheorem{lemma}{Lemma}[section]

\begin{document}
\title{
Selective model-predictive control for flocking systems.
}
\author[1]{Giacomo Albi \thanks{giacomo.albi@ma.tum.de, address: Boltzmannstr. 3, Garching bei M\"unchen,  D-85748, Germany}}
\author[2]{Lorenzo Pareschi \thanks{ lorenzo.pareschi@unife.it, address: via Machiavelli 35, Ferrara, IT-44121, Italy}}
\affil[1]{Department of Mathematics, TU M\"unchen}
\affil[2]{Department of  Mathematics and Computer Science, University of Ferrara}
\maketitle

\begin{abstract}
 In this paper the optimal control of alignment models composed by a large number of agents is investigated in presence of a selective action of a controller, acting in order to enhance consensus. Two types of selective controls have been presented: an homogeneous control filtered by a selective function and a distributed control active only on a selective set. As a first step toward a reduction of computational cost, we introduce a model predictive control (MPC) approximation by deriving a numerical scheme with a feedback selective constrained dynamics. Next, in order to cope with the numerical solution of a large number of interacting agents, we derive the mean-field limit of the feedback selective constrained dynamics, which eventually will be solved numerically by means of a stochastic algorithm, able to simulate efficiently the selective constrained dynamics. Finally, several numerical simulations are reported to show the efficiency of the proposed techniques.
\end{abstract}

\section{Introduction}
In this paper we focus on multi-agent systems subject to a velocity alignment dynamics \cite{CuckerSmale1, GP,HHK10,MoTa:13,olfati2007consensus,Vicsek12}, and influenced by an external control. The control term can be used to study the influence of an external agent on the system dynamics, or to enforce emergence of non spontaneous desired asymptotic states. This type of problems has risen large interest in several communities, see for example \cite{AHP13,BKF14,bofo13, BW14,CFPT12, MJ07,MNJD09}.

We will consider models of Cucker-Smale-type, \cite{CuckerSmale1, HHK10,MoTa:13} consisting of a second order system where an average process among binary interactions rules the alignment of the velocities, and where such interactions are weighted by a function of the relative distance among two agents.  Thus we consider a system of $N$ agents with position and velocity $(x_i,v_i)\in\mathbb{R}^{2d}$,  and initial datum $(x_i(0),v_i(0)) = (x_i^0,v_i^0)$,  which evolves according to 
as follows
\begin{equation}
\begin{aligned}\label{eq:CS1}
&\dot{x}_i=v_i,\\
&\dot{v}_i=\frac{1}{N}\sum_{i=1}^N H(x_i,x_j)(v_j-v_i), \qquad i=1,\ldots,N
\end{aligned}
\end{equation}
where $H(x,y) = H(|x-y|)$, is a general communication function weighting the alignment towards the mean velocity. The following theorem holds for systems of type \eqref{eq:CS1}, 
\begin{thm}{(Unconditional flocking).}\label{thm:CS}
Let us consider the system \eqref{eq:CS1}, where the communication function $H(x,y) = H(|x-y|)$ is assumed to be decreasing and such that
\[ 
(i)\quad H(r)\leq H(0)\leq 1,\qquad (ii)\quad \int_{0}^{\infty}H(r)\, dr = \infty,
\]
then the maximal diameter of the positions remains uniformly bounded by $d^\infty$, namely we have $\max_{ij} \| x_i(t)-x_j(t)\|\leq d^\infty<\infty$ for every $t\geq 0$, and the velocities converge exponentially fast towards the flocking state $v^\infty$, such that
\[
\|v_i(t)-v^\infty\|\leq \exp\left\{- H(d^\infty) t\right\}\max_{ij}\|v_i^0-v_j^0\|.
\]
\end{thm}
\noindent
We refer to \cite{CuckerSmale1, HHK10,MoTa:13} for proof of this result.  

The unconditional flocking can be retrieved also in the case conditions $(i)$-$(ii)$ do not hold, but particular properties on the initial state of the system should be satisfied, see \cite{CuckerSmale1}. In a general setting these conditions are not usually fulfilled, therefore a natural way to reach the flocking condition in system of type \eqref{eq:CS1}, requires the intervention on the system, by means of a designed control  or an external force.
This idea has been studied in several research fields, and from different perspectives, in particular for Cucker-Smale-type models,\cite{AHP13,bofo13,BKF14,FPR14,BW14,ABCK15}. Hence a natural framework for such problem consists in finding a control strategy ${\bf u}=(u_1,\ldots,u_N)\in\mathbb{R}^{d\times N}$ in the space of the admissible controls $\mathcal{U}$, solution of the following minimization problem, 
\begin{equation}
\begin{aligned}\label{eq:pbc0}
&\qquad\qquad\min_{{\bf u(\cdot)}\in\mathcal{U}} J_T({\bf u}(\cdot))=\frac{1}{2}\int_0^T\frac{1}{N}\sum_{j=1}^N\left(\|v_j-\bar v\|^2 +\Psi(u_j)\, \right)dt,
\end{aligned}
\end{equation}
where $\bar v$ is a desired velocity, $\Psi:\mathbb{R}^d\to\mathbb{R}_+$ a positive convex function, and subject to the dynamics
\begin{equation}
\begin{aligned}\label{eq:pbm0}
&\dot{x}_i=v_i,\\
&\dot{v}_i=\frac{1}{N}\sum_{j=1}^N H(|x_i-x_j|)(v_j-v_i) + u_i, \qquad i=1,\ldots,N.
\end{aligned}
\end{equation}

The type of control problem underlined by \eqref{eq:pbc0}-\eqref{eq:pbm0} is equivalent to assume the presence of a {\em policy maker} able to exert a control action on every single agent. 
From a modeling view point this choice can be justified in different ways, for example by considering a limited amount of resources, thus a governor cannot act simultaneously on every agent but it has to select a portion of the population in order to make its action relevant; alternatively a control active only on few agents, can be promoted taking into account a $\ell_1$-minimization in \eqref{eq:pbc0} , \cite{CFPT12,ABCK15}. 
On the other hand, if we do not consider enough restriction on the action of our policy maker, we can not exclude that the solution of the optimal control problem \eqref{eq:pbc0}-\eqref{eq:pbm0} requires the ability of controlling point-wisely every agent in a large set. 

In this paper will address the goal of controlling a large agents' ensemble when the action of the controllers is performed in selective way. We will give two different interpretation to the concept of selective control.

\begin{itemize}
\item[$(a)$] First we will consider {\em selectivity} as an intrinsic property of the system of filtering external actions, thus the control influences the flocking dynamics through a {\em selective function} $S(x,v,t)$, which depends on the state of the system. In particular we will assume the control to be equal for all agents, i.e. $u=u_1=\ldots = u_N$, and introduce a function $S(\cdot)$ modeling the propensity to accept the information coming from $u$.

\item[$(b)$] Second, we will define {\em selectivity } according to the membership to a particular set, and we will consider the control to be active only on the agents belonging to a given set $\Omega\subset\mathbb{R}^d\times\mathbb{R}^d\times[0,T]$. We will assume that such set can be defined according to some properties of the system and of the domain.
\end{itemize}

From the mathematical viewpoint to each model we will associate different control problems, in the first case we will assume a setting in terms of optimal control problems, \cite{AHP13,BW14,CFPT12}, in the second case we will introduce a type of differential game, making the implicit assumption that agents inside the selective set wish to optimize their own functional. \cite{DLR,DHL15}, for connection among this two different approach we refer to \cite{BFY:13}. 

Moreover, since we are interested also in the numerical investigation of the models, we will introduce a numerical strategy to reduce the computational cost. Indeed the numerical solution of  control problems for systems of type \eqref{eq:pbc0}-\eqref{eq:pbm0}, requires usually a tremendous computational efforts, due to the nonlinearities of the model and the large number of agents, \cite{Sontag1998aa}. 
Thus a first step towards the cost reduction is obtained numerically via model-predictive control (MPC), \cite{MayneMichalska1990aa}, when dealing with such complex system,  where instead of solving the above control problem over the whole time horizon, the system is approximated by an iterative solution over a sequence of finite time steps, \cite{AHP13,APZ14}.

A further approximation of the microscopic dynamics consists in the derivation of the so-called mean-field limit of the particle system, which describes the behavior of the system for large number of agents, \cite{MR2744704,CCH13,CCR11}. Such statistical description of the evolution of the microscopic system, has been recently coupled with the optimal control problem, in order to furnish a novel description of problems in terms of mean-field optimal control problems \eqref{eq:pbc0}-\eqref{eq:pbm0}, \cite{BFY:13,FS}.

In this direction our paper presents a simple selective approach to obtain greedy solutions to the mean-field optimal control problem, embedding into the mean-field dynamics the instantaneous controls obtained from the MPC procedure.

The paper is organized as follows, in Section \ref{sec:one} we described our modeling setting, deriving a approximated solution  through an model-predictive control strategy (MPC); in Section \ref{sec:two} we will derive formally a mean-field description of the constrained problem. Finally in Section \ref{sec:three} we present several numerical results showing the efficiency of the proposed technique.

\section{Selective control of flocking models}\label{sec:one}
We propose two different models and control setting to promote the alignment of a flocking system: a space homogeneous control filtered by a {\em selective function} measuring the influence on the dynamics, an heterogeneous control which activates on some specific agents once they belong to the {\em selective set}, $\mathcal{S}$.

\paragraph{Filtred control with selective function.}
We consider a Cucker-Smale-type model where a system of $N$ agents described by their position and velocity, $(x_i,v_i)\in\mathbb{R}^{2d}$ with initial datum $(x_i(0),v_i(0)) = (x_i^0,v_i^0)$, evolves accordingly to 
\begin{equation}
\begin{aligned}\label{eq:pbm1}
&\dot{x}_i=v_i,\\
&\dot{v}_i=\frac{1}{N}\sum_{j=1}^N H(|x_i-x_j|)(v_j-v_i)+u\Qs(x_i,v_i,t),\qquad i=1,\ldots,N
\end{aligned}
\end{equation}
where $H:\mathbb{R}_+\to[0,1]$, is a general communication function weighting the alignment towards the mean velocity, depending on the relative distance of the agents. 

The control term $u\in\mathbb{R}^d$ is include as an external intervation, whose action is multiplied by function $S:\mathbb{R}^{2d}\times[0,T]\to[0,1]$ a real-valued function, which tunes the influence of the control on the single agent according to its position and velocity.
We refer to $\Qs(x,v,t)$ as {\em selective function} and the term $u\Qs(x,v,t)$ as {\em filtered control}. Thus we define  $u\in\mathbb{R}^d$ as the solution of the following optimal control problem
\begin{equation}
\begin{aligned}\label{eq:pbc1}
&u^*(\cdot)=\arg\min_{u:[0,T]\to\Omega} J_T(u(\cdot))=\frac{1}{2}\int_0^T\left(\frac{1}{N}\sum_{j=1}^N\|v_j-\bar v\|^2 + \nu\|u\|^2\right)\, dt,
\end{aligned}
\end{equation}
constrained to the dynamics of  \eqref{eq:pbm1}, where $\bar v\in\mathbb{R}^d$ is a desired velocity, and $\nu>0$ a regularization parameter the space of admissible controls. For simplicity in the formulation \eqref{eq:pbc1} we consider a least--square type cost functional, but other choices can be considered.

\begin{rmk}
Observe that a-priori $S(x,v,t)$ could be un-known by the controller, which would implies a different modeling setting accounting possible uncertainties. Instead we will assume that the controller has a perfect knowledge of the selective function $S(x,v,t)$, which might be eventually his belief on the effectiveness of the control action.
\end{rmk}

\paragraph{Pointwise control with selective set.}

In a different setting we consider the Cucker-Smale-type model  \eqref{eq:pbm0} where each agent is controlled directly, thus we have a system of $N$ agents with initial datum $(x_i(0),v_i(0)) = (x_i^0,v_i^0)\in\mathbb{R}^{2d}$ evolving according to
\begin{equation}
\begin{aligned}\label{eq:pbm2}
&\dot{x}_i=v_i,\\
&\dot{v}_i=\frac{1}{N}\sum_{j=1}^N H(|x_i-x_j|)(v_j-v_i)+u_i,\qquad i=1,\ldots,N,
\end{aligned}
\end{equation}
where ${u_i(t)}\in\mathbb{R}^d$ is the strategy of every agents. We  assume that the strategy of the agent $i$ is active only if $i\in\mathcal{S}(t)$ for every $i=1,\ldots,N$, and where $\mathcal{S}$ is defined for every $t\in[0,T]$as follows
\begin{align}\label{eq:SelSet}
\mathcal{S}(t) = \left\{i\in\{1,\ldots,N\}\, | (x_i(t),v_i(t))\in\Omega(t)\subset \mathbb{R}^d\times \mathbb{R}^d \right\},
\end{align} 
where $\Omega(t)$ is the selective set defined on the phase space at time $t$. 
Therefore only the dynamics of a portion of the total agents can be influenced, in order to steer whole system towards the target velocity $\bar v$.

Differently from the previous model we set up the problem in a differential game setting. 
Thus the strategy of each agent is defined through the following set of minimization problems, 
\begin{equation}\label{eq:pbc2b}
{u^*_i(t)}=
\begin{cases}
\displaystyle\arg\min_{{ u_i}(\cdot):[t,T]\to\Omega} J^i_T({u_i(\cdot)})=\dfrac{1}{2}\displaystyle\int_t^{T}\left(\frac{1}{N}\sum_{j=1}^N\|v_j-\bar v\|^2 + {\nu}\|u_i\|^2\right)\, dt&\quad\textrm{ if } i \in\mathcal{S}(t),\\
0 &\quad\textrm{ otherwise,}
\end{cases}
\end{equation}
for $i=1,\ldots,N$, where each agent wants to minimizes  its own functional $J^i_T$, as soon as at time $t$ he belongs to $\mathcal{S}(t)$. Therefore each agent strategy at time $t$ is the solution of an equilibrium problem solved among the agents in the set $\mathcal{S}$. 

\begin{rmk}
We observe that this problem can be formulated also in terms of an optimal control problem, where the aim is to find a vector ${\bf u} = (u_1,\ldots,u_N)\in\mathbb{R}^{d\times N}$ solution to the following minimization problem,
\begin{equation}\label{eq:pbc2a}
{\bf u^*(\cdot)}=\arg\min_{{\bf u}:[0,T]\to\Omega^N} J_T({\bf u}(\cdot))=\dfrac{1}{2}\displaystyle\int_0^T\left(\frac{1}{N}\sum_{j=1}^N\|v_j-\bar v\|^2 + \frac{\nu}{N}\sum_{\ell\in\mathcal{S}(t)}\|u_\ell\|^2\right)\, dt,
\end{equation}
constrained to \eqref{eq:pbm2}. Note that at variance with problem \eqref{eq:pbc1},  functional \label{eq:pbc2a} carries an additional  difficulty given by implicit  dependency of $\mathcal{S}(t)$ with respect to the agent dynamics.


\end{rmk}

\subsection{Model predictive control}
We introduce now a numerical technique based on model predictive control (MPC), also called receding horizon strategy in the engineering literature, in order to reduce the computational cost of the optimal control problem, \cite{CaBo:04,MayneMichalska1990aa}.
This procedure is in general only suboptimal with respect to the global optimal solution of problems \eqref{eq:pbm1}--\eqref{eq:pbc1}, and  \eqref{eq:pbm2}--\eqref{eq:pbc2a}, nonetheless we will show that also in the simplest setting the solution of the MPC furnishes an instantaneous feedback control, which is a consistent discretization of a first order approximation of the optimal control dynamics.

Let us consider the time sequence $0=t_0<t_1<\ldots<t_M=T$, a discretization of the time interval $[0,T]$, where  $\Delta t = t_{n}-t_{n-1}$, for all $n=1,\ldots, M $ and $t_M = M\Delta t $. Then we assume the control to be constant on every interval $[t_n,t_{n+1}]$, and defined as a piecewise function, as follows
\begin{align}\label{eq:mpcctrl}
u(t) = \sum_{n=0}^{M-1} u^n\chi_{[t_n,t_{n+1}]}(t),
\end{align} 
where $\chi(\cdot)$ is the characteristic function of the interval $[t_n,t_{n+1}]$.
In general  {\em model predictive control}  strategies solve a finite horizon open-loop optimal control problem predicting the dynamic behavior over a predict horizon $t_m\leq t_M$, with initial state sampled at time $t$ (initially $t = t_0$), and computing the control on a control horizon $t_c\leq t_m$. 

Since our goal is to derive instantaneous control strategies, in what follows we will consider a reduced setting $t_m=t_c=t+\Delta t$, and taking in to account a first order discretization of the optimal control problem \eqref{eq:pbm1}-\eqref{eq:pbc1}.

\subsubsection{Instantaneous filtered control}
Let us introduce a full discretization of the system  \eqref{eq:pbm1} through a forward Euler scheme, and we solve the minimization problem  \eqref{eq:pbc1} via MPC strategy on every time frame $[t_n, t_n+\Delta t]$. 
Thus the reduced optimal control problem reads
\begin{equation}\label{eq:disc_functional}
\min_{u^n}J_{\Delta t}= \frac{1}{2N}\sum_{j=1}^N \|v_j^{n+1} -{\bar v}\|^2 +\frac{\nu}{2} \|u^n\|^2,
\end{equation}
subject to 
\begin{equation}
\begin{aligned}\label{eq:Fwmicr}
x_i^{n+1} &= x_i^n+\Delta t v_i^n,\\
v_i^{n+1} &= v_{i}^n +  \frac{\Delta t}{N}\sum_{j=1}^N H(|x^n_i-x^n_j|)(v_j^n-v_i^n)+ \Delta t u^n\Qs(x_i^n,v_i^n,t^n),
\end{aligned}
\end{equation}
for all $i=1,\ldots, N$, and $\u^n\in\Omega\subset \mathbb{R}^d$. The MPC aims at determining the value of the control ${u}^n$ by solving for the known state $(x^n_i,v^n_i)$ a (reduced) optimization problem on $[t_n,t_{n+1}]$ in order to obtain the new state  $(x_i^{n+1},v_i^{n+1})$. This procedure is reiterated until $n \Delta t = T$ is reached.   
 In this way it is possible to reduce the complexity of the initial problem  \eqref{eq:pbm1}-\eqref{eq:pbc1},
 to an optimization problem in a single variable ${u}^n$. 
Therefore we introduce the compact notation 
$H^n_{ij} = H(|x_i^n-x_j^n|)$, and $\Qs^n_i = \Qs(x_i^n,v_i^n,t^n)$, where for every $i$, $\p_i$ is the associated lagrangian multiplier of $\v_i$, and we define the discrete Lagrangian  $\mathcal{L}_{\Delta t}=\mathcal{L}_{\Delta t}(\v,{\u}^n,\p)$, such that
\vskip -0.556cm
\begin{equation}\label{eq:Lagr}
\begin{aligned}
&\mathcal{L}_{\Delta t}=&J_{\Delta t} (\v,\u^n)+\frac{1}{N}\sum_{j=1}^N \p_j\cdot\left(\v_j - {v}_j^n -\frac{\Delta t}{N}\sum_{\ell=1}^{N}H^n_{j \ell}({v}^n_\ell-{v}^n_j)-\Delta t {\u^n}\Qs^n_j\right). \\
\end{aligned}
\end{equation}
%
%
Computing the gradient of \eqref{eq:Lagr} with respect to each component of $\v_i$ and $\u^n$ for every $i=1,\ldots,N$, we obtain the following first order optimality conditions
\begin{subequations}
\begin{align}
& \v_i-\bar v + \p_i = 0,\label{eq:optctrl}\\
& \nu\u^n +\frac{\Delta t}{N} \sum_{j=1}^N\p_j\Qs^n_j= 0.\label{eq:optadj}
\end{align}
\end{subequations}
This approach allows to express explicitly the control as feedback term of the state variable, indeed plugging expression \eqref{eq:optctrl} into \eqref{eq:optadj}, we have that for every $n= 0,\ldots,M-1$
\begin{equation}\label{eq:Dfwdcontr1}
\begin{aligned}
& \u^n = -\frac{\Delta t}{\nu N} \sum_{j=1}^N(\v_j-\bar v)\Qs^n_j.
\end{aligned}
\end{equation}
Substituting in the discretized system \eqref{eq:Fwmicr} the expression obtained in \eqref{eq:Dfwdcontr1}, the feedback controlled system results
\begin{equation}\label{eq:Dfwdcontr2}
\begin{aligned}
& v^{n+1}_i=v^n_i+\frac{\Delta t}{N}\sum_{j=1}^{N}H^n_{ij}(v^n_j-v^n_i)+\frac{(\Delta t)^2}{\nu N} \sum_{j=1}^N(\bar v-v_j^{n+1})\Qs^n_j\Qs^n_i,\qquad i=1,\ldots,N,
\end{aligned}
\end{equation}
where the action of the control  is substituted by an implicit term representing the relaxation towards the desired velocity $\bar v$.
Note that in this implicit formulation the action of the control is lost for $\Delta t\to 0$, since it is expressed in terms of $O(\Delta t ^2)$. Thus, in order to rewrite the system as a consistent time discretization of the original control problem  is necessary to assume the following scaling on the regularization parameter,  $\nu=\Delta t \kappa$, and we revert to the system into an explicit form,
thus we obtain
\begin{equation}
\begin{aligned}\label{fullDsystem1}
&v^{n+1}_i= v^n_i+\frac{\Delta t}{N}\sum_{j=1}^NH^n_{ij}(v^n_j-v^n_i)+\frac{\Delta t}{N\kappa+\Delta t\sum_{j=1}^N(\Qs^n_i)^2}\sum_{j=1}^N\left( \bar v-v^n_j\right)\Qs^n_i\Qs^n_j+O(\Delta t^2).
\end{aligned}
\end{equation}
where we have omitted $O(\Delta t^2)$ terms. We leave the details of the derivation of the forward system in Appendix \ref{app:A}.

Hence system \eqref{fullDsystem1} represents a consistent discretization of the following  dynamical system, 
\begin{equation}
\begin{aligned}\label{eq:MFmodel1}
&\dot x_i= v_i,\\
&\dot v_i= \frac{1}{N}\sum_{j=1}^NH(|x_i-x_j|)(v_j-v_i)+\frac{1}{\kappa N}\sum_{j=1}^N\left( \bar v-v_j\right)\Qs(x_j,v_j,t)\Qs(x_i,v_i,t).
\end{aligned}
\end{equation}
where the control term is expressed by a steering term acting as an average weighted by the {\em selective fuction} $\Qs(\cdot,\cdot)$.
\begin{rmk}
Let us assume $\Qs(\cdot,\cdot,\cdot)\equiv 1$, and defining $m(t)$ the mean velocity of the system, then we have
\begin{equation}
\begin{aligned}
\dot m&= \frac{1}{\kappa}(v_d-m),\qquad m(0) = m_0.
\end{aligned}
\end{equation}
which admits the explicit solution, $m(t) = (1-e^{-t/k})\bar v + e^{-t/k}m_0$. Therefore, for $t\to\infty$, $m(t)=\bar v$. Thus, in this case, the feedback control is able to control only the mean of the system but not to assure the global flocking state, note.
\end{rmk}

\subsubsection{Instantaneous pointwise control}
Similarly to previous section we introduce a full discretization of the system through a forward Euler scheme of the optimal control problem  \eqref{eq:pbm2}-\eqref{eq:pbc2b} on every time frame $[t_n, t_n+\Delta t]$. Then the reduced optimal control problem reads
\begin{equation}\label{eq:pbc2d}
{u^*_i(t^n)}=
\begin{cases}
\displaystyle\arg\min_{{ u^n_i}} \dfrac{1}{2}\left(\frac{1}{N}\sum_{j=1}^N\|v^{n+1}_j(u^n)-\bar v\|^2 + {\nu}\|u^n_i\|^2\right)&\quad\textrm{ if } i \in\mathcal{S}(t^n),\\
0 &\quad\textrm{ otherwise,}
\end{cases}
\end{equation}
where the solution is easily retrieved by differentiation with respect to $u_i^n$, for every $i=1,\ldots,N$. Thus we have
\begin{equation}\label{eq:brs}
{u^*_i(t^n)}=
\begin{cases}
\displaystyle  -\frac{\Delta t}{\Delta t^2 +\nu}(v^{n}_i-\bar v) + \frac{\Delta t^2}{\Delta t^2 +\nu}\sum_{j=1}^NH_{ij}^n(v^n_j-v_i^n)&\quad\textrm{ if } i \in\mathcal{S}(t^n),\\
0 &\quad\textrm{ otherwise.}
\end{cases}
\end{equation}

In order to rewrite the system as a consistent time discretization of the original control problem we scale the regularization parameter,  $\nu=\Delta t \kappa$, and plugging the control \eqref{eq:brs} into the discretized dynamics, we obtain
\begin{equation}
\begin{aligned}\label{fullDsystem2b}
&v^{n+1}_i= v^n_i+\frac{\Delta t}{N}\sum_{j=1}^NH^n_{ij}(v^n_j-v^n_i)+\frac{\Delta t}{\kappa+\Delta t}(\bar v - v^{n}_i)\chi_{\Omega^n}(x^n_i,v^n_i) + O(\Delta t^2).
\end{aligned}
\end{equation} 
where we have omitted $O(\Delta t^2)$ terms and $\chi_{\Omega^n}$ is the characteristic function defined on the selective set $\Omega(t^n)$. 
Hence system \eqref{fullDsystem2b} is a consistent discretization of 
\begin{equation}
\begin{aligned}\label{eq:MFmodel2}
&\dot x_i= v_i,\\
&\dot v_i= \frac{1}{N}\sum_{j=1}^NH(|x_i-x_j|)(v_j-v_i)+\frac{1}{\kappa}(\bar v -v_i)\chi_{\Omega(t)}(x_i,v_i).
\end{aligned}
\end{equation}
 Note that at variance with respect to the previous case, the control is acting pointwisely on every single agent as a steering term towards the desired state. In the case of $H\equiv 1$, and $\Omega(t)\equiv\mathbb{R}^d\times\mathbb{R}^d$, for any $t\in[0,T]$ it can be easily shown that the velocities converge to the desired flocking state $\bar v$, for $t\to\infty$ and for any $\kappa>0$.
 \begin{rmk}
Let us remark that performing the MPC strategy on single time interval $[t^n,t^{n+1}]$ for the  optimal control problem \eqref{eq:pbc2a}, gives us the following discrete functional,
\begin{equation}\label{eq:disc_func2}
\min_{{\bf u}^n}J_{\Delta t}= \frac{1}{2N}\sum_{j=1}^N \|v_j^{n+1} -{\bar v}\|^2 +\frac{\nu}{2N} \sum_{\ell\in\mathcal{S}^n}^N\|u_\ell^n\|^2.
\end{equation}
Writing the discrete Lagrangian and computing its variations with respect to each components of ${\bf u}^n$ and ${v^{n+1}}$, gives us to the following system
\begin{align}\label{eq:Ffwd2d}
& v^{n+1}_i=
v^n_i+\Delta t\sum_{j=1}^{N}H^n_{ij}(v^n_j-v^n_i)-\frac{\Delta  t^2}{\nu N}(v_i^{n+1}-\bar v)\chi_{\Omega^n}(x^n,v^n).
\end{align}
Thus, by reverting to the explicit version of \eqref{eq:Ffwd2d} we obtain the same feedback control system \eqref{fullDsystem2b} with instantaneous control \eqref{eq:brs}. Therefore we have that the suboptimal controls recovered via model predictive control on the single horizon, respectively for \eqref{eq:pbc2a} and  \eqref{eq:pbc2b} are equivalent, \cite{DHL15, AHP13}.
\end{rmk}

\section{Mean-field limit for the controlled flocking dynamics}\label{sec:two}
We want to give a mean-field description of models \eqref{eq:MFmodel1} and \eqref{eq:MFmodel2}, thus let us write the constrained flocking system in the following general form
\begin{equation}
\begin{aligned}\label{eq:MFmodel}
&\dot x_i= v_i,\\
&\dot v_i= \mathcal{H}[f^N](x_i,v_i) + \mathcal{K}[f^N](x_i,v_i),
\end{aligned}
\end{equation}
where we introduced  the empirical probability measures
\begin{equation}\label{eq:emp}
f^N(x,v,t) = \frac{1}{N}\sum_{i=1}^N\delta(x-x_i(t))\delta(v-v_i(t)),
\end{equation}
representing the particle density at time $t$ with position and velocity $(x,v)\in\mathbb{R}^d\times\mathbb{R}^d$.
Moreover we defined the general operator $\mathcal{H}[\cdot]$ as follows
\begin{equation}
\begin{aligned}
\mathcal{H}[f](x,v) :=\psi(x,v)*f=\int_{\mathbb{R}^{2d}}H(|x-y|)(w-v)f(y,w,t)~dydw, 
\end{aligned}
\end{equation}
 with $\psi(x,v) := vH(|x|)$, and  where the operator $\mathcal{K}[\cdot]$ indicates in general the control term. Thus for different the types of instantaneous controls we derived in the previous sections, we have respectively
\begin{align}
&\mathcal{K}_\xi[f](x,v) := S(x,v,t)\xi(x,v,t) = \displaystyle \frac{1}{\kappa} S(x,v,t)\int_{\mathbb{R}^{2d}} (\bar v - w)S(y,w,t)f(y,w,t)\ dy dw,\label{xi}
\\
&\mathcal{K}_\zeta[f](x,v) := \zeta(x,v,t) = \frac{1}{\kappa} (\bar v - v)\chi_{\Omega(t)}(x,v)\label{zeta}.
\end{align}

Let us first derive formally the {\em mean-field limit } of system \eqref{eq:MFmodel}. We consider the empirical measures $f^N(t)$ defined in \eqref{eq:emp}, and a test function $\phi\in\mathcal{C}_{0}^1(\mathbb{R}^{2d})$, thus we compute 
\begin{align*}
\frac{d}{dt}\left\langle f^N(t),\phi\right\rangle=&\frac{1}{N}\sum_{i=1}^{N}\frac{d}{dt}\phi(x_i(t),v_i(t))=\frac{1}{N}\sum_{i=1}^{N}\nabla_x\varphi(x_i(t),v_i(t))\cdot v_i(t))\\
&\quad +\frac{1}{N}\sum_{i=1}^{N}\nabla_v\varphi(x_i(t),v_i(t))\cdot \left[\mathcal{H}[f^N](x_i(t),v_i(t)) + \mathcal{K}[f^N](x_i(t),v_i(t))\right],
\end{align*}
where $\langle\cdot,\cdot\rangle$ denotes the the integral in $x,v$ over the full $\mathbb{R}^{2d}$.
Collecting all the terms and integrating by parts in $(x,v)$ we recover the following weak formulation
\begin{align*}
\frac{d}{dt}\left\langle f^N,\phi\right\rangle=& -\left\langle v\cdot \nabla_x f^N, \phi\right\rangle-\left\langle \nabla_v\cdot  (\mathcal{H}[f^N]f^N), \phi\right\rangle-\left\langle \nabla_v\cdot  (\mathcal{K}[f^N]f^N),\phi \right\rangle.
\end{align*}
Rewriting the main expression we have
\begin{align*}
\left\langle \frac{\partial}{\partial t}f^N+ v\cdot \nabla_x f^N+\nabla_v\cdot  \mathcal{H}[f^N]f^N+\nabla_v\cdot  \mathcal{K}[f^N]f^N,\phi\right\rangle=0,
\end{align*}
and thus the strong form reads
\begin{align*}
 \frac{\partial}{\partial t}f^N+ v\cdot \nabla_x f^N+\nabla_v \cdot \mathcal{H}[f^N]f^N+\nabla_v\cdot  \mathcal{K}[f^N]f^N=0.
\end{align*}
%
%
%
%
Hence assuming that for $N\to \infty$ the limit $f^N\to f$ exists, where $f=f(x,v,t)$ is a probability density on $\mathbb{R}^{2d}$,
we obtain the following integro-differential PDE equation of the Vlasov-type,
\begin{equation}\label{eq:MFkinetic1}
\begin{aligned}
\partial_t f + v\cdot\nabla_x f&=-\nabla_v\cdot\left(\mathcal{H}[f]f\right)-\nabla_v\cdot\left(\mathcal{K}[f]f\right),
\end{aligned}
\end{equation}
as the mean-field limit of system \eqref{eq:MFmodel}.

In what follows we show some classical results of the rigorous derivation of the mean-field limit, restricting ourself to the control expressed by $\mathcal{K}_\xi[\cdot]$ in equation \eqref{xi} with selective function $S(x,v,t)$. Eventually we discuss the case of a general $\mathcal{K}[\cdot]$.
 
\paragraph{Stability result in presence of selective function.}
Let us consider $\mathcal{K}_\xi[f](x,v)$ defined as in \eqref{xi}, for this case we give sufficient conditions in order to prove the mean-field limit \eqref{eq:MFkinetic1}, (i.e. see hypothesis of Theorem 4.11 in \cite{CCR11}). To this end let us first introduce the following definition
\begin{definition}{(Wasserstein 1-distance).}\label{def:W}
Let $f,g\in\mathcal{P}_1(\mathbb{R}^d\times\mathbb{R}^d)$, be two Borel probability measures. Then the {\em Wasserstein} distance of order $1$ between $f$ and $g$ is defined as 
\begin{equation}
d_1(f,g) := \inf_{\pi\in\Pi}\left\{\int_{\mathbb{R}^{2d}\times\mathbb{R}^{2d}}|p_1-p_2|\ d\pi(p_1,p_2)\right\}
\end{equation}
\end{definition}
where the infimum is computed over the set of transference plans $\Pi$ between $f$ and $g$, i.e. among the probability measures $\pi$ in the product space $\mathbb{R}^{2d}\times\mathbb{R}^{2d}$ with marginals $f$ and $g$.

We further define $\mathcal{P}_{c}(\mathbb{R}^d\times\mathbb{R}^d)$ the subset of probability measures of compact support on $\mathbb{R}^d\times\mathbb{R}^d$, with finite first moment, and we define the non-complete metric space $\mathcal{A}:= \mathcal{C}([0,T],\mathcal{P}_c(\mathbb{R}^d\times\mathbb{R}^d)$ endowed with the Wasserstein 1-distance.
Moreover we introduce the set of functions $\mathcal{B}:= \mathcal{C}([0,T],\mathrm{Lip}_{loc}(\mathbb{R}^d\times\mathbb{R}^d,R^d)$, which are locally Lipschitz with respect to $(x,v)$, uniformly in time. 
Therefore let us consider the operator $\mathcal{F}[\cdot]:\mathcal{A}\to\mathcal{B}$, such that 
\[
\mathcal{F}[f](x,v):= \mathcal{H}[f](x,v) +  \mathcal{K}_\xi[f](x,v),  
\]
then we state the following
\begin{lemma}\label{lemm:stability}
Let $\psi(x,v)$, and $S(x,v,t)$ be locally Lipschitz and bounded, and $f,g\in\mathcal{A}$, such that $\textrm{\em supp}(f_t)\cup\textrm{\em supp}(g_t)\subseteq B_{r_0}$, for every $t\in[0,T]$ and for a given radius $r_0>0$.
Then for any ball $B_{r}\subset\mathbb{R}^d\times\mathbb{R}^d$, there exists a constant $C:=C(r_0,r)$ such that
\begin{subequations}
\begin{align}
&\max_{t\in[0,T]} \textrm{\em Lip}_r(\mathcal{F}[f])\leq C,\label{eq:hypo1}\\
& \|\mathcal{F}[f]-\mathcal{F}[g]\|_{L^{\infty}(B_r)}\leq C d_1(f,g),\label{eq:hypo2}
\end{align}
\end{subequations}
where $\textrm{\em Lip}_r(\mathcal{F}[f])$ denotes the  Lipschitz constant in the ball $B_r\subset\mathbb{R}^d\times\mathbb{R}^d$.
\end{lemma}
\begin{proof}
Let us first define estimate
\begin{align*}
|\mathcal{F}[f](x_1,v_1)&-\mathcal{F}[f](x_2,v_2)|\leq \int_{\Rdd} \left|\psi(x_1-y,v_1-w)-\psi(x_2-y,v_2-w) \right| f(y,w)\ dydw\\
&+ \frac{1}{\kappa}\left|S(x_1,v_1)-S(x_2,v_2)\right|\int_{\Rdd}\left|\bar v - w\right||S(y,w)| f(y,w)\ dydw\\
&\leq \left(\|\psi\|_{\textrm{Lip}} +  M_r\|S\|_{\textrm{Lip}}  \right)|(x_1,v_1)-(x_2,v_2)|
\end{align*}
where for the sake of brevity we omit the dependency on $t$, and where $M_r:=M(r,r_0,|S|,\kappa)$, where $r>0$ is such that  $\textrm{\em supp}(f_t)\cup\textrm{\em supp}(g_t)\subseteq B_{r}$, then we have  \eqref{eq:hypo1}.
Let us now introduce the optimal transference plan $\pi$ between $f$ and $g$, in the sense of Definition \ref{def:W}, and having defined  $\xi(x,v) = (\bar v -v)S(x,v)/\kappa$, which is again locally Lipschitz thanks to the boundedness of $S$, then for any $(x,v)\in B_r\subset \Rdd$ we have
\begin{align*}
&\mathcal{F}[f](x,v)-\mathcal{F}[g](x,v) = \int_{\Rdd} \psi(x-y,v-w) f(y,w)\ dydw-\int_{\Rdd}\psi(x-z,v-u)  f(z,u)\ dzdu\\
&\qquad\qquad+ S(x,v)\left(\int_{\Rdd} \xi(y,w) f(y,w)\ dydw-\int_{\Rdd}\xi(z,u)  f(z,u)\ dzdu\right)\\
&= \int_{\Rddd} \left(\psi(x-y,v-w) -\psi(x-z,v-u)\right) + S(x,v)\left( \xi(y,w) -\xi(z,u)\right)\pi(y,w;z,u)
\end{align*}
Thus, taking the absolute value we have
\begin{align*}
&|\mathcal{F}[f](x,v)-\mathcal{F}[g](x,v)| \leq \int_{\Rddd} \left|\psi(x-y,v-w)-\psi(x-z,v-u) \right|\ d\pi(y,w;z,u)\\
&\quad+ |S(x,v)|\int_{\Rddd}\left|\xi(y,w)- \xi(z,u) \right|\ d\pi(y,w;z,u)
\leq \left(\|\psi\|_{\textrm{Lip}} +  |S| \|\xi\|_{\textrm{Lip}} \right)d_1(f,g),
\end{align*}
which implies \eqref{eq:hypo2}.
\end{proof}

In this case the results of Lemma \ref{lemm:stability} are sufficient to satisfy the hypothesis of Theorem 4.11 in \cite{CCR11}, in this way existence, uniqueness and stability of measure solutions for model \eqref{eq:MFkinetic1} are assured. The remarkable consequence of this theorem is the stability of the solutions in the Wasserstein 1-distance, which gives us a rigorous derivation of the kinetic equation \eqref{eq:MFkinetic1} as the limit of the a large number of agents of the system of ODEs \eqref{eq:MFmodel} in the Doubrushin's sense, for further details see \cite{CCR11}. 
\begin{rmk}
In the case of control with {\em selective  set} $\Omega(t)$, the previous results are not valid anymore, because it carries the discontinuous function $\chi_\Omega (x,v)$, which is not locally Lipschitz, therefore we need more refined estimates to prove a stability result. From the modeling view point a possible strategy consists in considering a mollified version of the $\chi_\Omega(x,v)$ in order to gain enough regularity,  \cite{CCR11,CCH13,H13}. More refined result for the mean-field limit  have been shown in the case of discontinuous kernels and they might be extended to this case,~\cite{CCH13}.
\end{rmk}
\section{Numerical simulations}\label{sec:three}
One of the main difficulty in the numerical solutions of kinetic models of type \eqref{eq:MFkinetic1}, arises in the approximation of the interaction operators, $\mathcal{H}[\cdot]$ and $\mathcal{F}[\cdot]$, which requires usually a huge computational efforts.
In order to reduce the computational complexity we use a fast numerical algorithm based on the approximation of the interaction operator through a Boltzmann-like equation, we leave further details of this procedure in Appendix \ref{app:B} and we refer to \cite{AlbiPareschi2013ab}.

We perform the simulations for $(x,v)\in\mathbb{R}^2\times\mathbb{R}^2$, defining an initial data $f(x,v,t=0) = f_0(x,v)$ normally distributed in space, with center in zero and unitary variance, and in velocity, uniformly distributed on a circumference of radius $5$. Our goal is to enforce alignment with respect to the desired velocity $\bar v=(1,1)^T$. The evolution of the kinetic equation \eqref{eq:MFkinetic1} is evaluated up to final time $T = 4$, with $\Delta t  = 0.01$ for the time discretization, considering $N_s = 5\times10^5$ sampled particles and scaling parameter $\varepsilon=\Delta t$.

Hence we consider the mean-field model \eqref{eq:MFkinetic1}, with the standard communication function,
$$H(x,y) = (1+|x-y|^2)^{-\gamma},$$ with $\gamma=10$, with this choice of $\gamma$, then the hypothesis of Theorem \ref{thm:CS} are not satisfied, and consequently the unconditional flocking is not guaranteed a-priori. 

We report in Figure \ref{fig:F1} the initial data and the final state reached at time $T=4$, depicting the spatial density $\rho(x,t) = \int_{\mathbb{R}^d} f(x,v,t)\ dv$ and showing at each point $x\in\mathbb{R}^2$ the value of the flux, represented by the following vector quantity $\rho u(x,t) = \int_{\mathbb{R}^d}f(x,v,t)\ dv$.
Note that in the right-hand side figure the flocking state is not reached, and the density is spreading around the domain following the initial radial symmetric distribution of the velocity field.
\begin{figure}[thpb]
\centering
\includegraphics[scale=0.325]{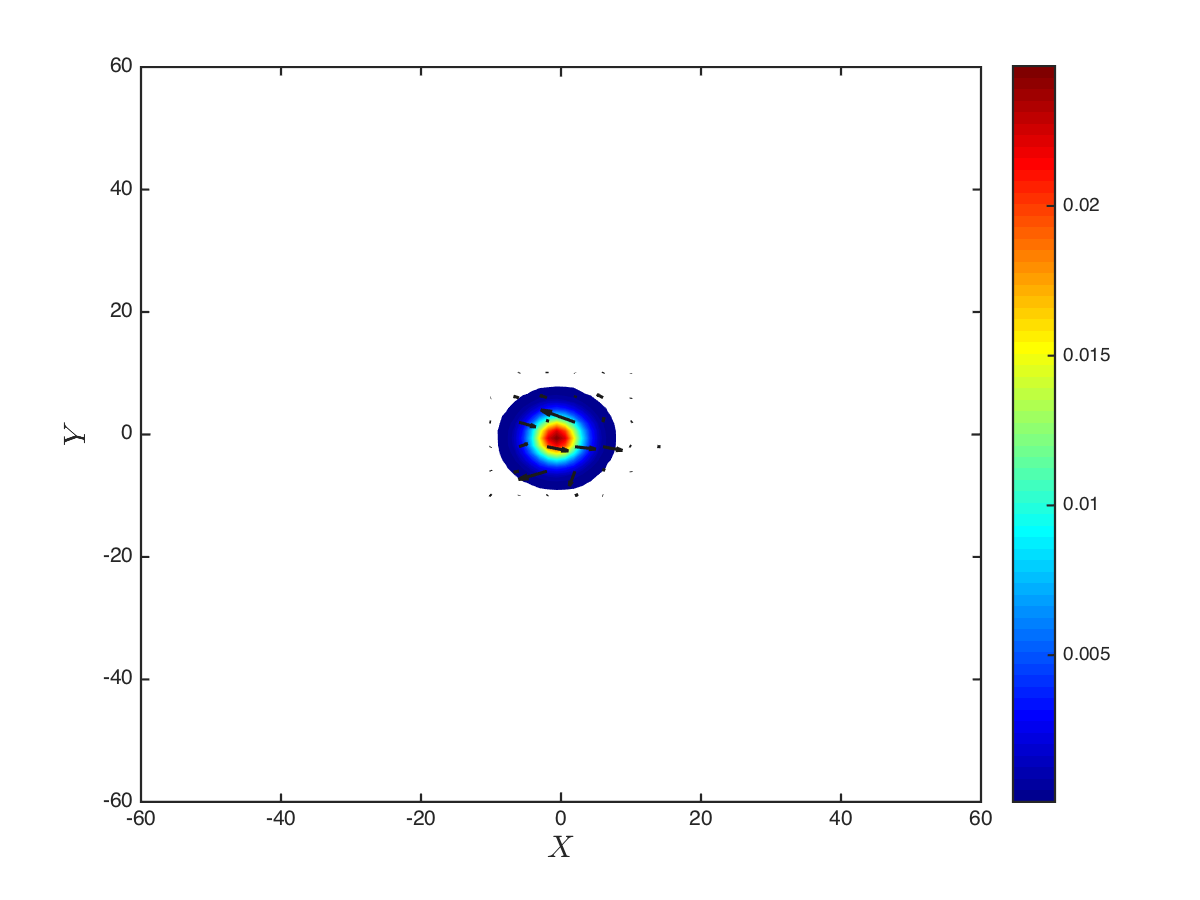}
\includegraphics[scale=0.325]{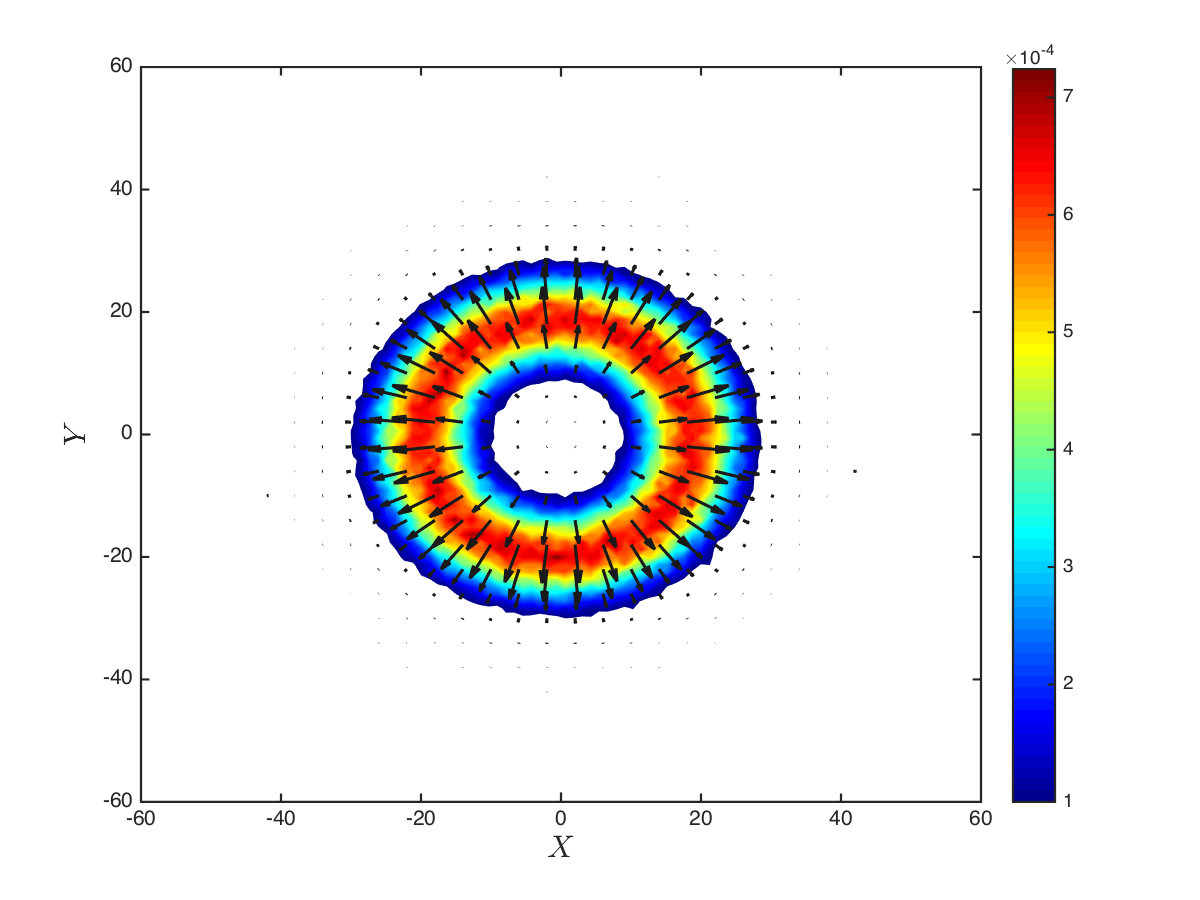}
\caption{On the left-hand side initial data, on the right-hand side configuration of the solution of \eqref{eq:MFkinetic1} at time $T=4$. In absence of control the density spread around the domain without reaching the consensus.}\label{fig:F1}
\end{figure}

Starting from this initial example we want to stabilize the evolution testing the performances of the different control policies in presence of a selective function and in the case of a selective set.

\subsection{Localized stabilization}
We compare the two control approaches in the case of a selective control only capable of acting on a confined ball of the space domain. Hence, we define $$B_R= \{ x\in\mathbb{R}^2\ s.t.\ |x|^2\leq R^2\},$$ and we study the evolution of the mean-field equation \eqref{eq:MFkinetic1} in presence of controls of type \eqref{xi}  and \eqref{zeta}. 
\paragraph{Test 1a: filtered control}\label{test1}
We study the evolution of the system in presence of a selective control, where the selective function is $S(x,v) = \chi_{B_R}(x)$, and the filtered control defined by
\begin{equation}\label{eq:ctrl1}
\mathcal{K}_\xi[f](x,v):=S(x,v)\xi(t) = \frac{\chi_{B_R}(x)}{\kappa}\int_{B_{R}}(\bar v- w)f(y,w,t)\ dy dw\l.
\end{equation}
Moreover in order to compare the behavior of the action of the selective control we define respectively the running cost, $L[\cdot,\cdot]$ and the total cost $C_T$ as
\begin{align}\label{eq:funcJT} 
L[f,\xi](t):= \int_{\mathbb{R}^4}|v-\bar v|^2f(x,v,t)\ dx dv + {\kappa}|\xi(t)|^2, \textrm{ and } C_T  :=\int_0^TL[f,\xi](t)\ dt.
\end{align}
The numerical results in Figure \ref{Fig:F2} shows the higher influence of the control for increasing value of $R$, i.e. for longer influence of the control on the density, and decreasing value of the penalization parameter $\kappa$.
\begin{figure}[thpb]
\centering
\begin{tabular}{@{}c@{\hspace{1mm}}c@{\hspace{1mm}}c@{\hspace{1mm}}c@{}}
\hline
&$\kappa = 4$ & $\kappa = 1$& $\kappa = 0.25$\\
\hline
\sidecap{$R=5$} 
&
\includegraphics[trim=30 10 40 20,clip,width=0.25\textwidth]{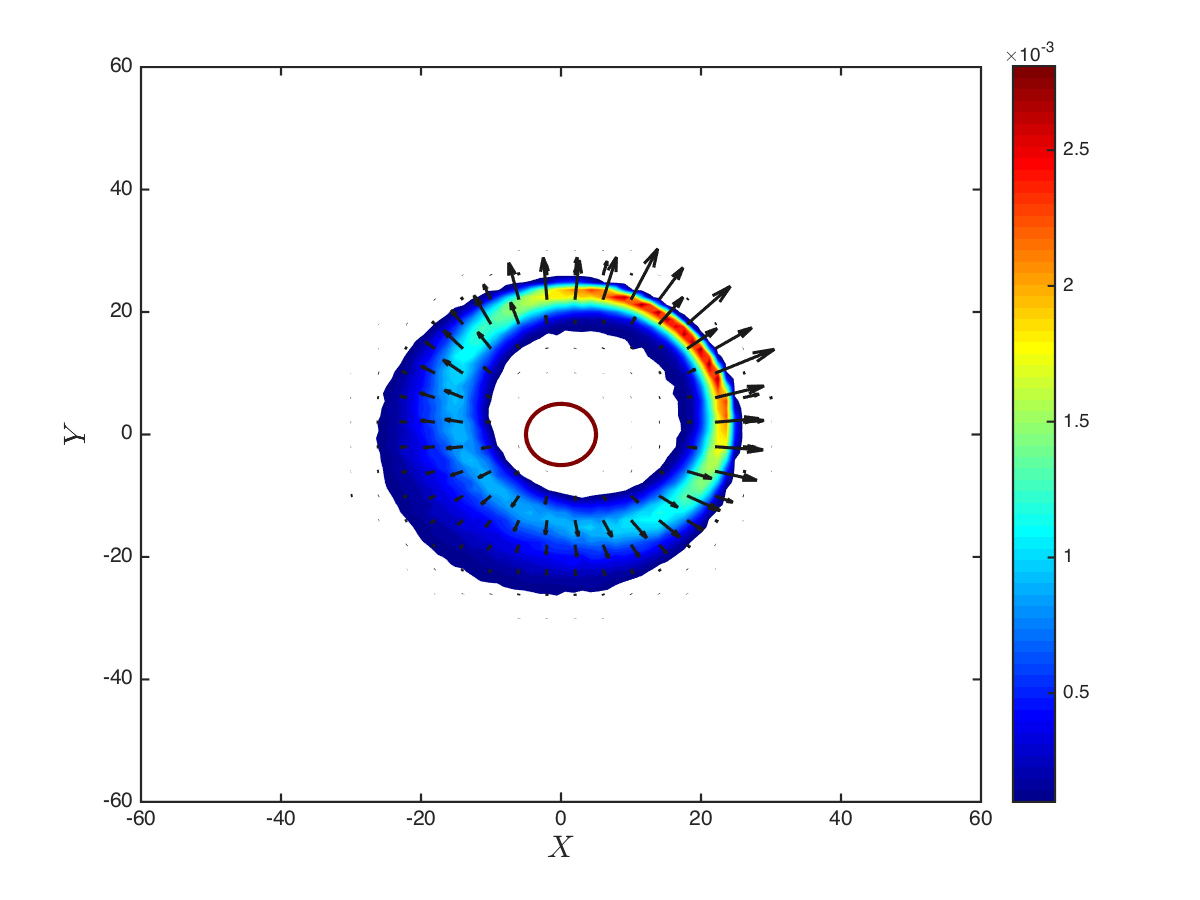}&
\includegraphics[trim=30 10 40 20,clip,width=0.25\textwidth]{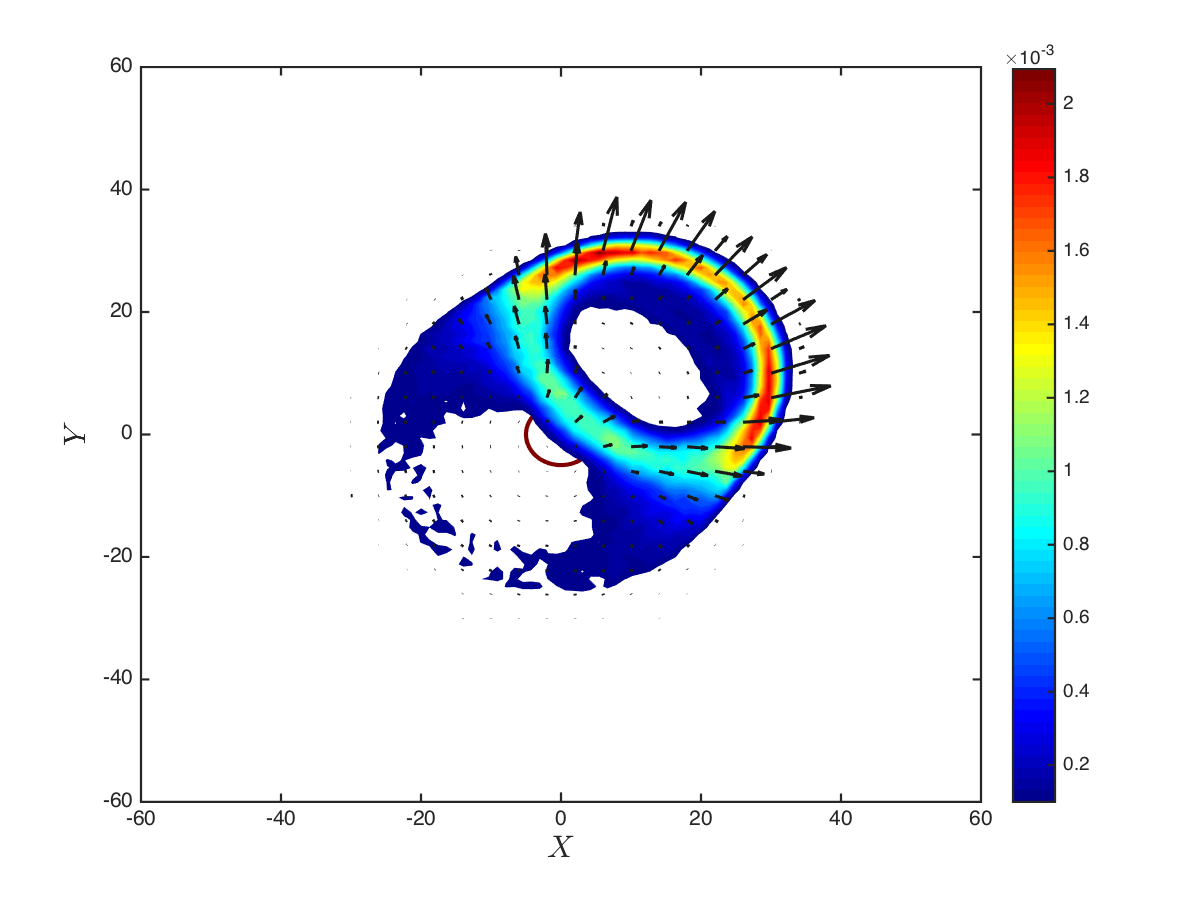}&
\includegraphics[trim=30 10 40 20,clip,width=0.25\textwidth]{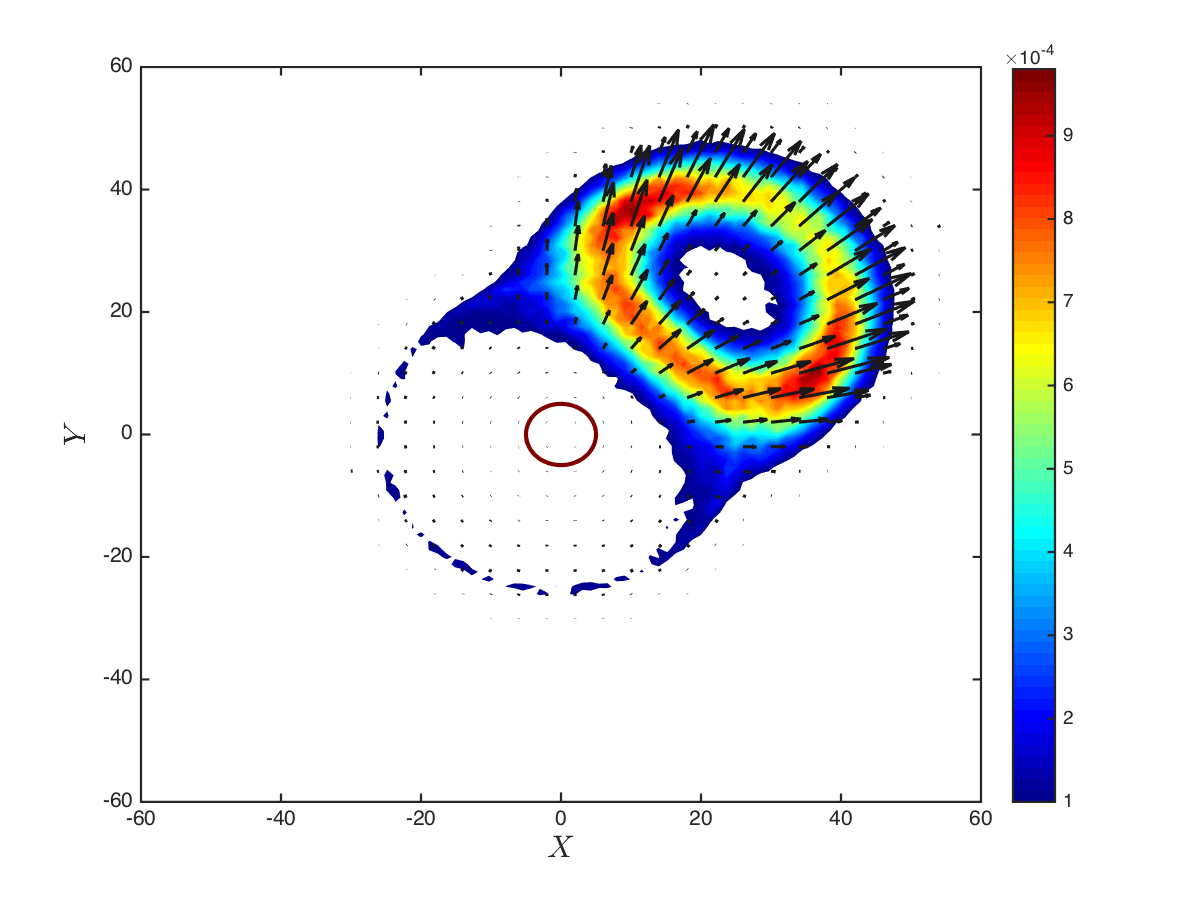}
\\
& $C_T=2.7908   $  & $C_T= 1.3954   $ & $C_T=0.8095$ \\
\hline
 \hline
\sidecap{$R=10$}
&
\includegraphics[trim=30 10 40 20,clip,width=0.25\textwidth]{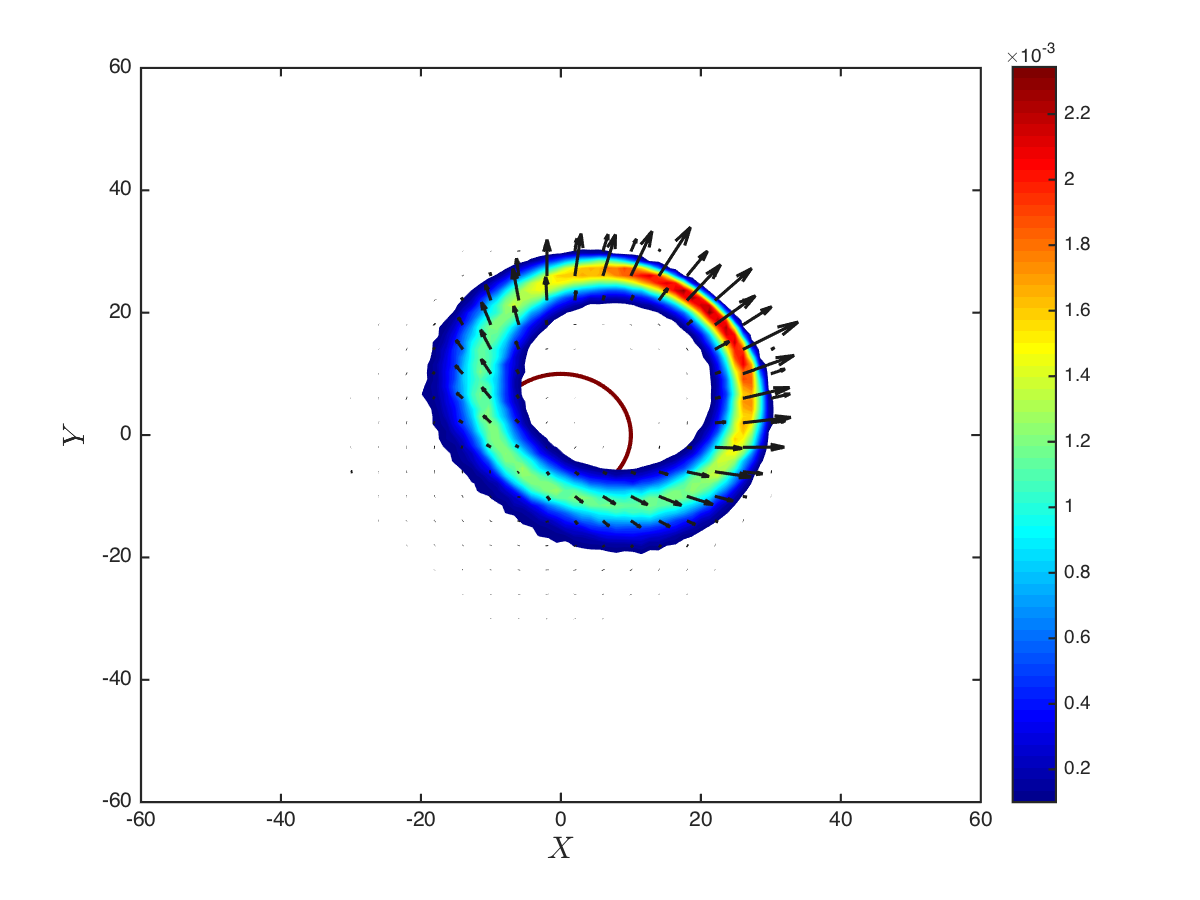}&
\includegraphics[trim=30 10 40 20,clip,width=0.25\textwidth]{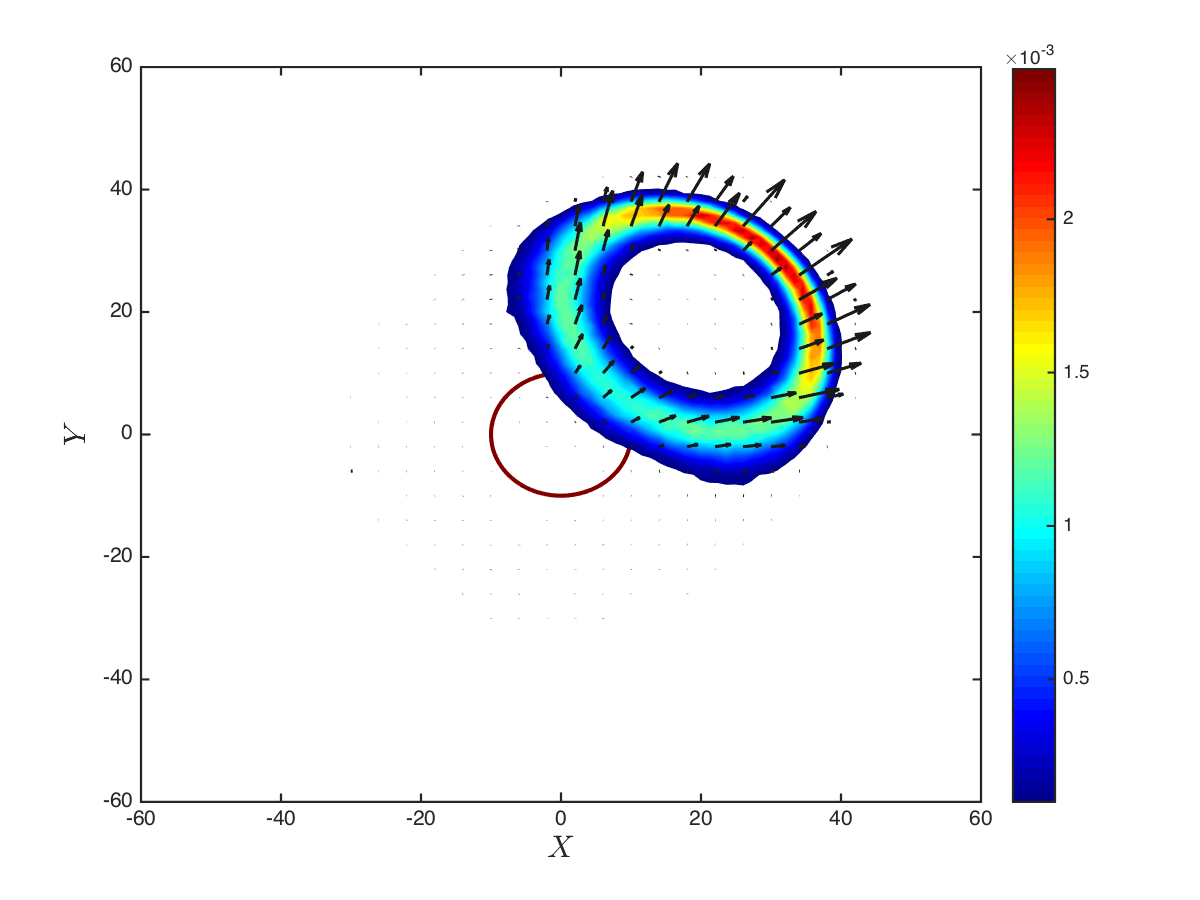}&
\includegraphics[trim=30 10 40 20,clip,width=0.25\textwidth]{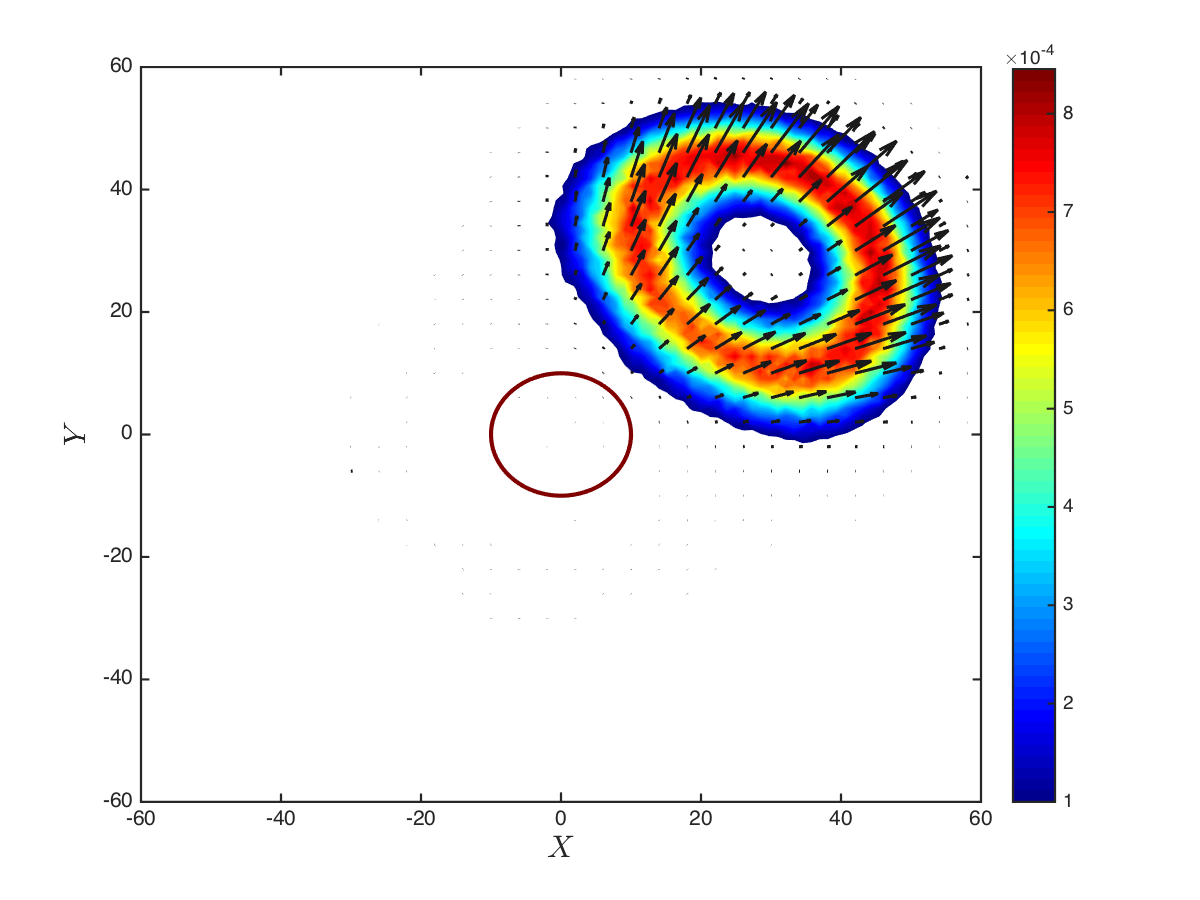}
\\
& $C_T=  2.7472 $  & $C_T= 1.4141$ & $C_T=0.7992 $ \\
 \hline
\end{tabular}

\caption{(Test 1a): Final solution at time $T=4$ of equation \eqref{eq:MFkinetic1} with control acting through a selective function $S(x,v)=B_R(x)$. The top and bottom pictures represent the action of the control, respectively  for $R=5$, and $R=10$, and for different values of penalization parameter $\kappa$.  Value of the cost functional \eqref{eq:funcJT} are reported below each simulation. }\label{Fig:F2}
\end{figure}

In Figure \ref{fig:F4} we additionally explore the range of parameters $(R,\kappa)\in[0,50]\times[0,10]$, with respect to the following quantities: ${A}_{R,\kappa}:=\int |v-\bar v|^2f(x,v,T)\ dxdv$ measuring the alignment at final time for $T = 4$, and the following cost ${C}_{R,\kappa}:=\frac{1}{\kappa T}\int^T_0\int |v-\bar v|^2\Qs(x,v)f(x,v,t)\ dxdv\ dt$.
 
\begin{figure}[thpb]
\centering
\includegraphics[scale=0.35]{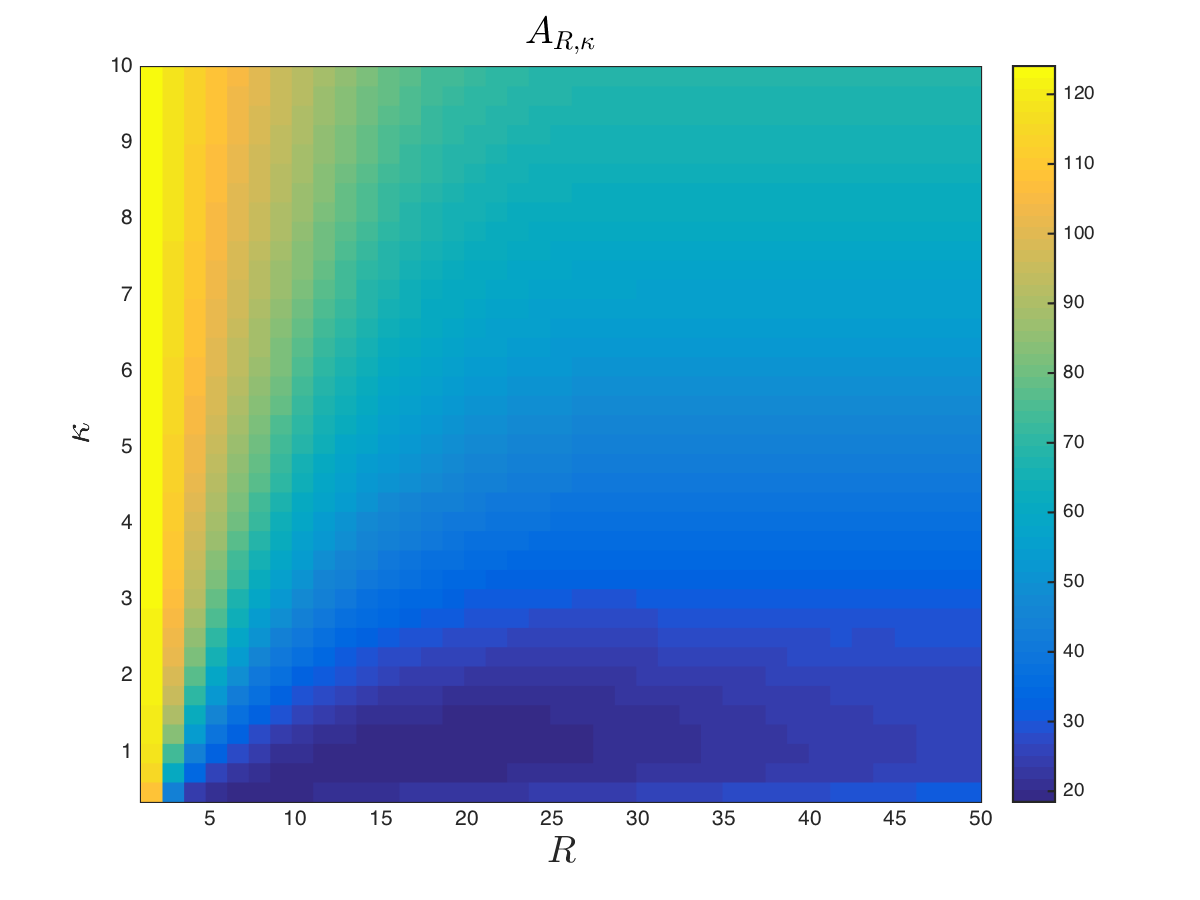}
\includegraphics[scale=0.35]{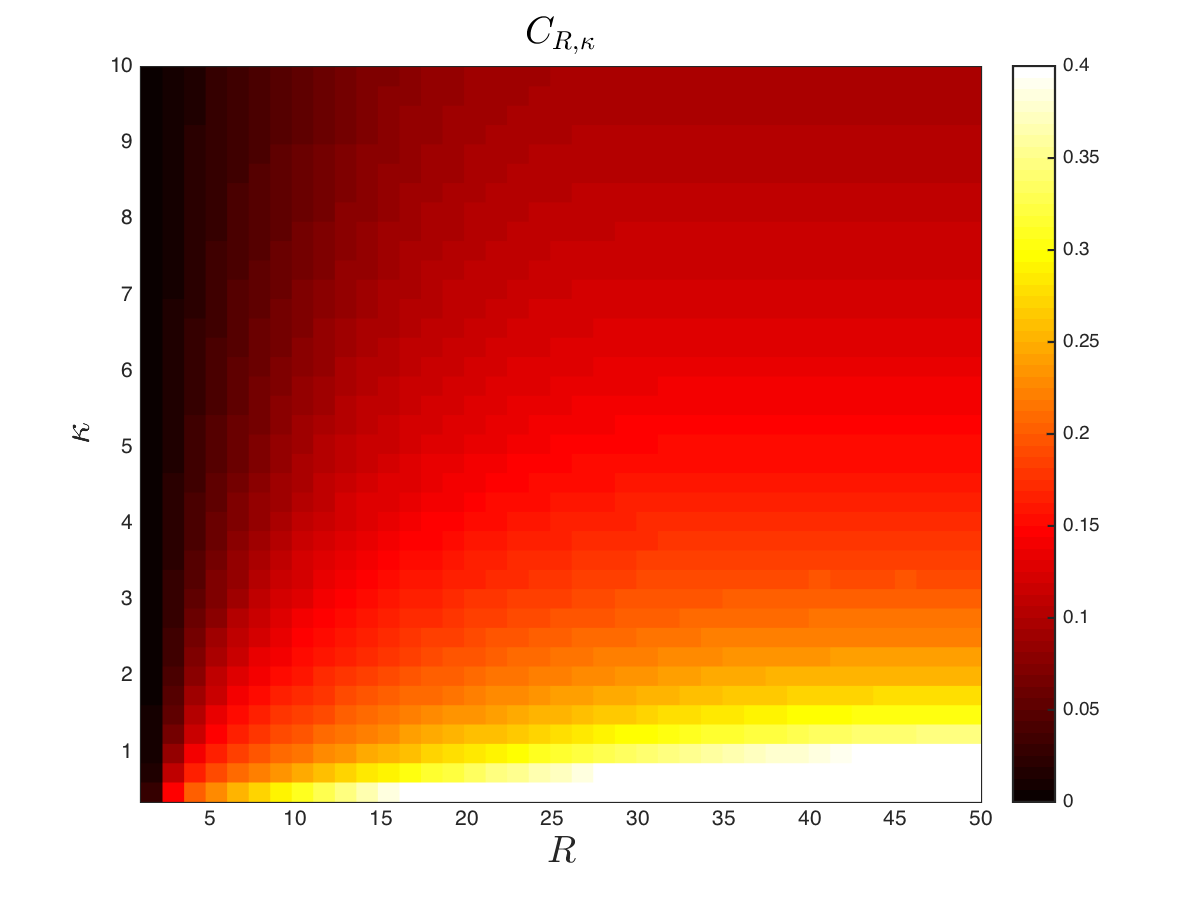}
\caption{We set the region of parameters to be $(R,\kappa)\in[0.1,50]\times[0.1,10]$, and we study (left) the convergence to final flocking state, 
${A}_{R,\kappa}$, with $T = 4$, darker regions represent closer state to global alignment;  (right) the total cost of the control ${C}_{R,\kappa}$.
}\label{fig:F4}
\end{figure}
\paragraph{Test 1b: pointwise control.}\label{test2}
We now study the action of a control in presence of the selective set $\mathcal{S}(t)= B_R(x)$, therefore  the control is active only on the density $f(x,v,t)$ which is included inside the ball of radius $R$, and is defined as follows
\begin{equation}\label{eq:ctrl2}
\mathcal{K}_\zeta[f](x,v) :=\zeta(x,v,t)= \frac{1}{\kappa}(\bar v- v)\chi_{B_R}(x).
\end{equation}
and similarly to the previous test we associate the following continuous cost functional
\begin{align}\label{eq:funcJT2} 
C_T  :=\int_0^T L[f,\zeta](t)\ dt, \quad \textrm{ with }\quad  L[f,\zeta](t)= \int_{\mathbb{R}^4}\left(|v-\bar v|^2 + {\kappa}|\zeta(x,v,t)|^2\right)f(x,v,t)\ dx dv.
\end{align}
We report in Figure \ref{Fig:F4} the final solution of \eqref{eq:MFkinetic1} at time $T= 4$, under the influence of the control $\zeta(x,v,t)$, which shows a stronger influence with respect to the previous case in Test \ref{test1}, indeed the alignment is obtain in almost every case,
thanks to the local influence of the control which is acting on every point inside the ball $B_R$.

In Figure \ref{Fig:F5} we report the evolution of the running cost $L[f,\cdot](t)$ defined respectively in \eqref{eq:funcJT}, and \eqref{eq:funcJT2}.  The plots show a not-monotone decay, until a plateau is reached, which occurs once the control is no-more active and the velocity field reaches an equilibrium.
\begin{figure}
\centering
\begin{tabular}{@{}c@{\hspace{1mm}}c@{\hspace{1mm}}c@{\hspace{1mm}}c@{}}
\hline
&$\kappa = 4$ & $\kappa = 1$& $\kappa = 0.25$\\
\hline
\sidecap{$R=5$} 
&
\includegraphics[trim=30 10 40 20,clip,width=0.25\textwidth]{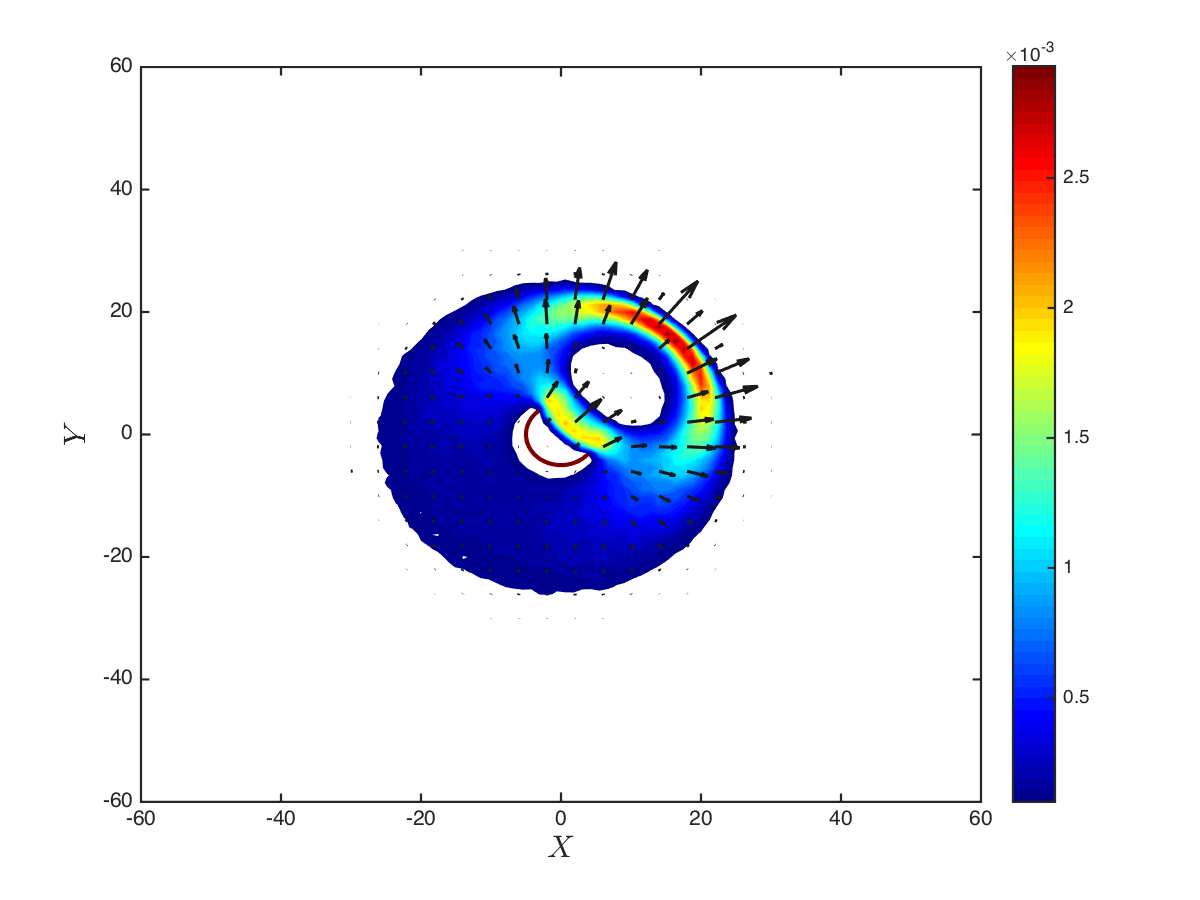}&
\includegraphics[trim=30 10 40 20,clip,width=0.25\textwidth]{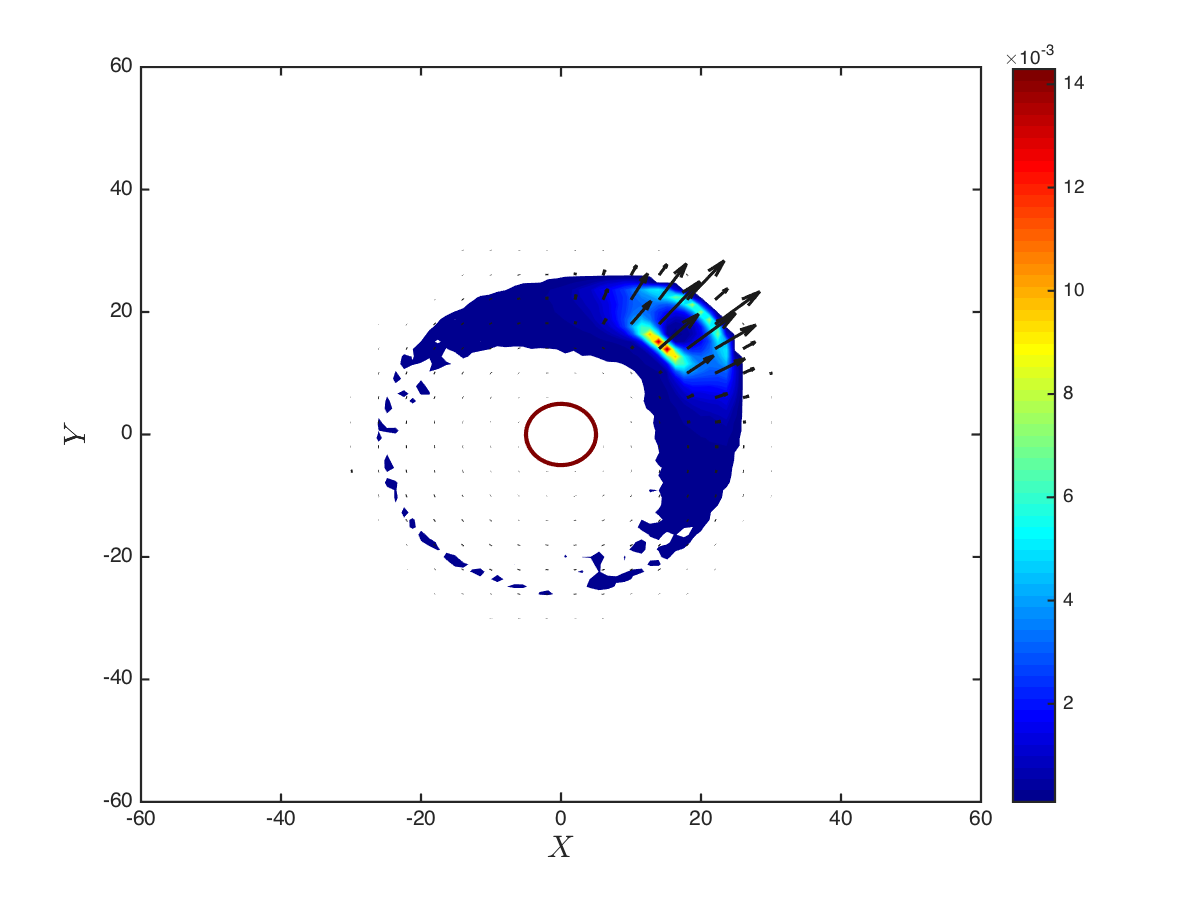}&
\includegraphics[trim=30 10 40 20,clip,width=0.25\textwidth]{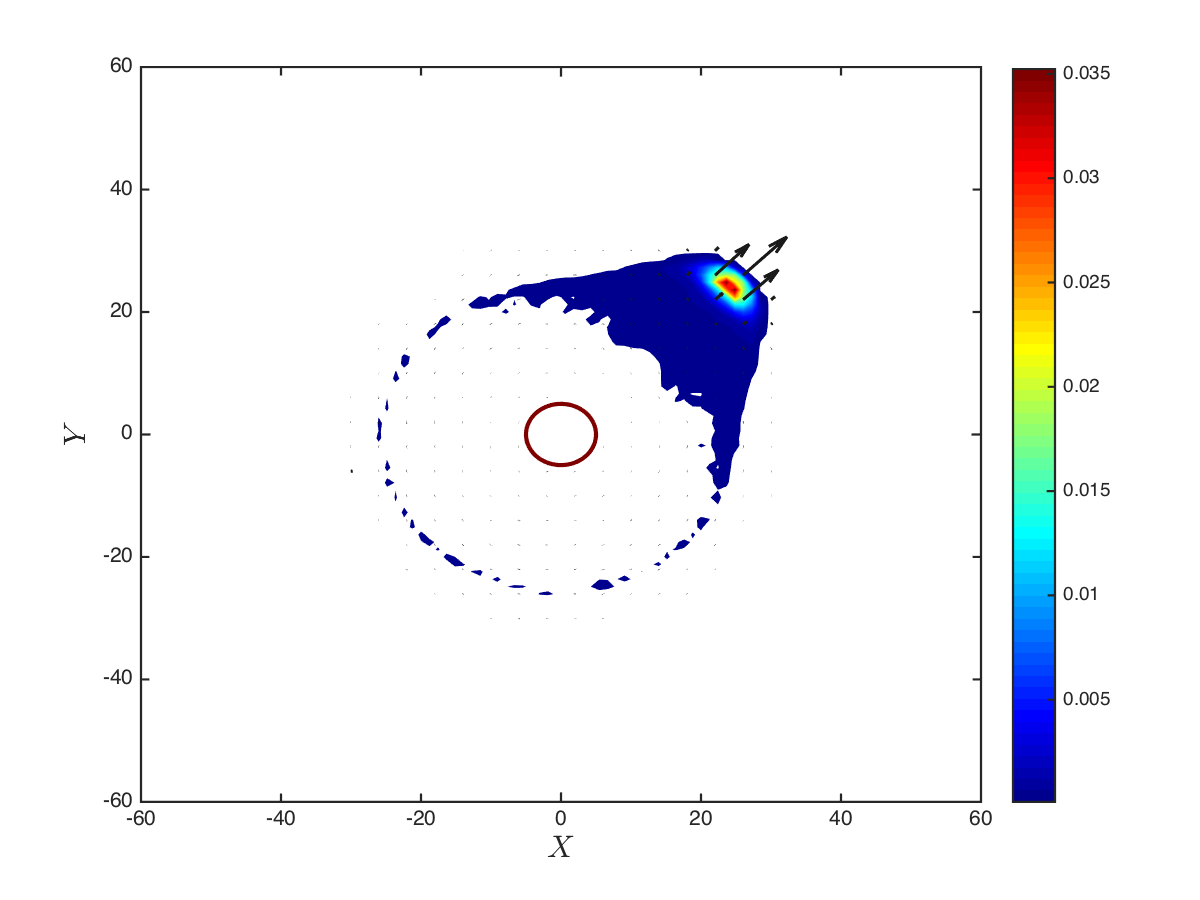}
\\
& $C_T= 2.2574 $  & $C_T= 0.9856 $ & $C_T= 0.4898 $ \\
\hline
 \hline
\sidecap{$R=10$}
&
\includegraphics[trim=30 10 40 20,clip,width=0.25\textwidth]{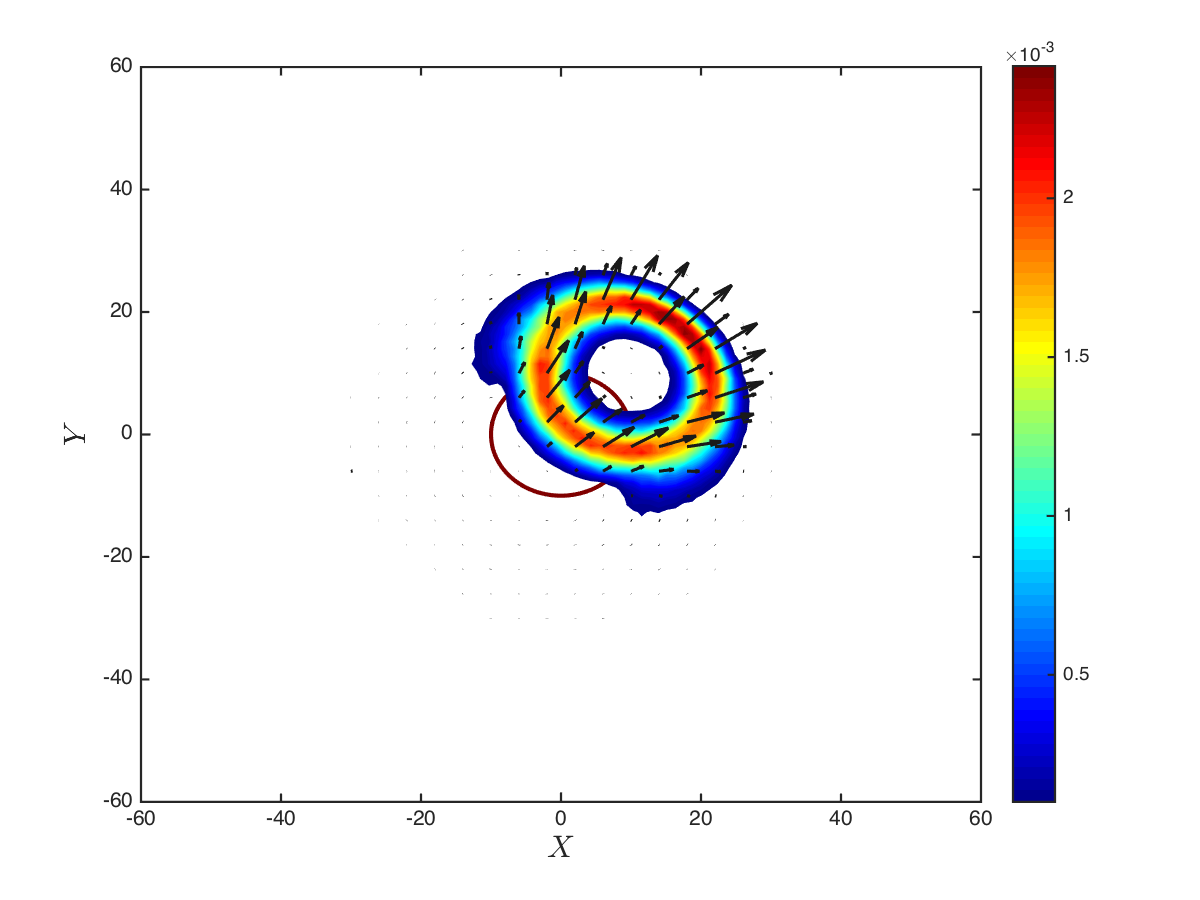}&
\includegraphics[trim=30 10 40 20,clip,width=0.25\textwidth]{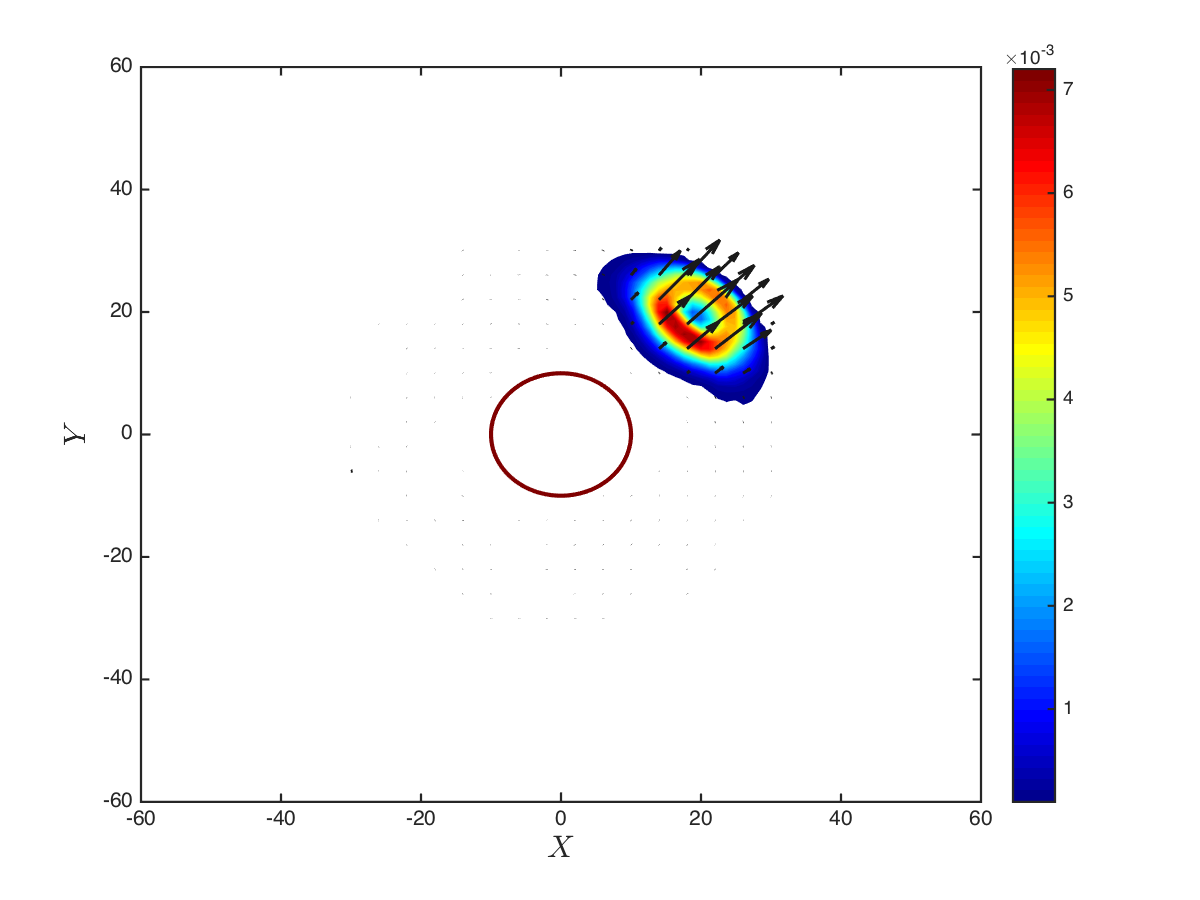}&
\includegraphics[trim=30 10 40 20,clip,width=0.25\textwidth]{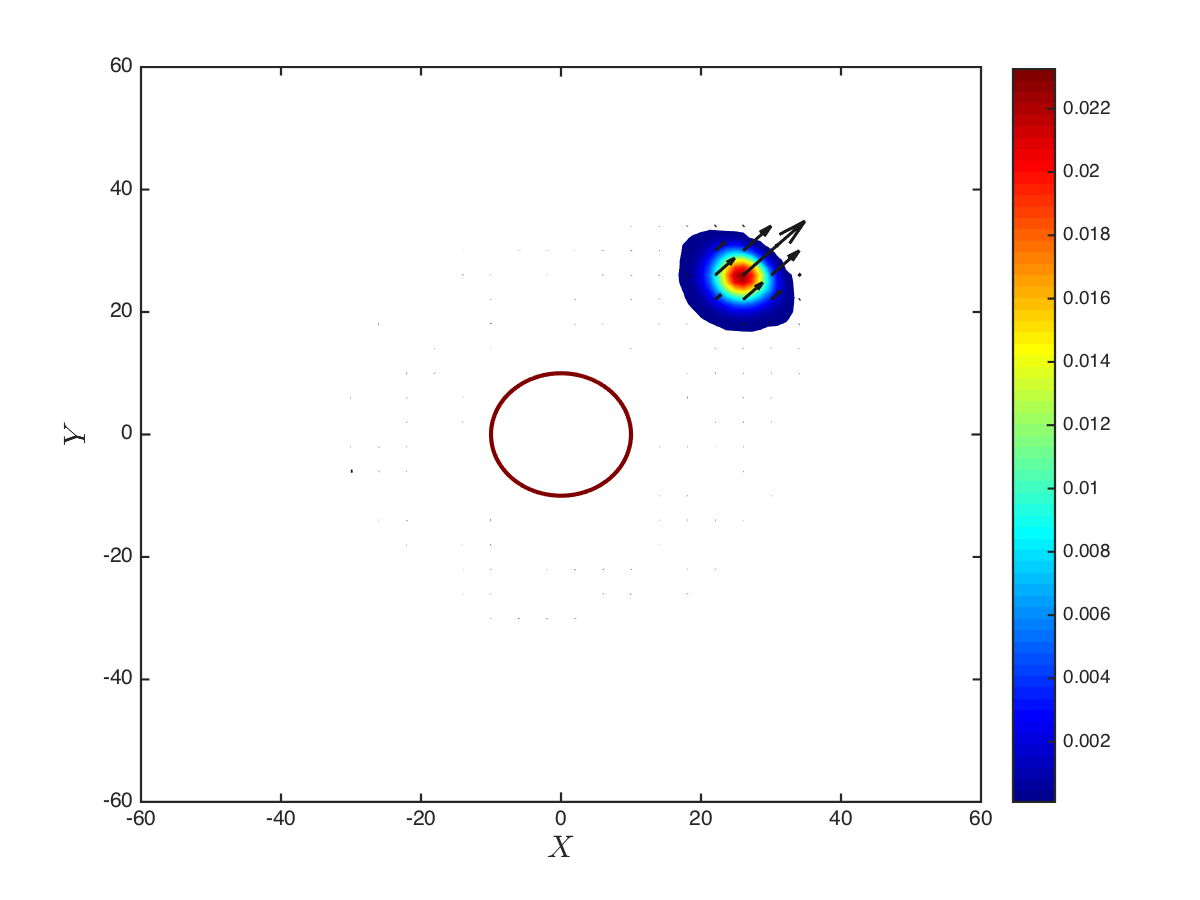}
\\
& $C_T=2.0704  $  & $C_T=1.0815   $ & $C_T= 0.3550 $ \\
 \hline
\end{tabular}
\caption{(Test 1b): Final states at time $T=4$ of the density $f(x,v,T)$ with control acting only on the set $B_R(x)$. The top and bottom pictures represent the action of the control, respectively  for $R=5$, and $R=10$, and for different values of penalization parameter $\kappa$.  Value of the cost functional \eqref{eq:funcJT2} are reported below each simulation.}\label{Fig:F4}
\end{figure}

\begin{figure}
\centering
\begin{tabular}{@{}c@{\hspace{1mm}}c@{\hspace{1mm}}c@{\hspace{1mm}}c@{}}
\hline
&$\kappa = 4$ & $\kappa = 1$& $\kappa = 0.25$\\
\hline
\sidecap{$L[f,\cdot](t)$} 
&
\includegraphics[trim=30 10 40 20,clip,width=0.25\textwidth]{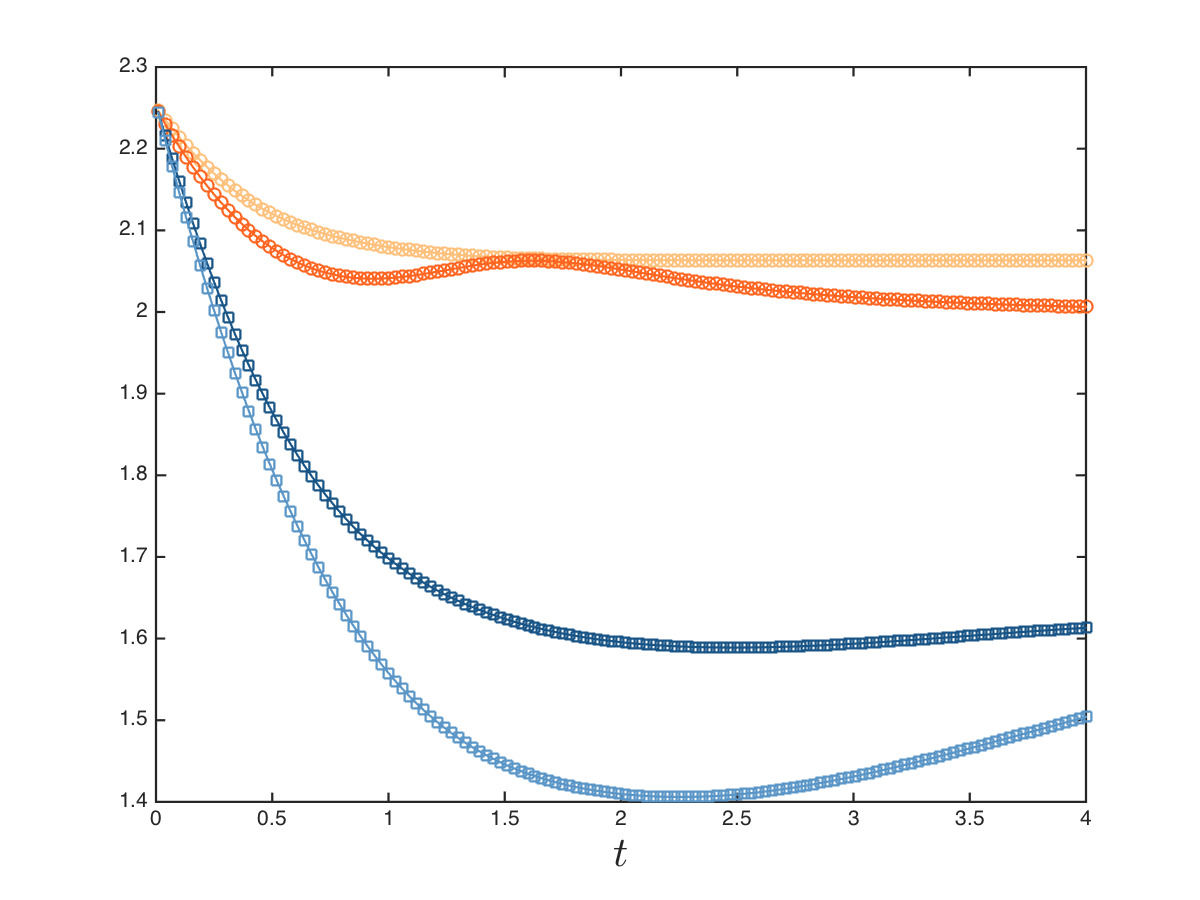}&
\includegraphics[trim=30 10 40 20,clip,width=0.25\textwidth]{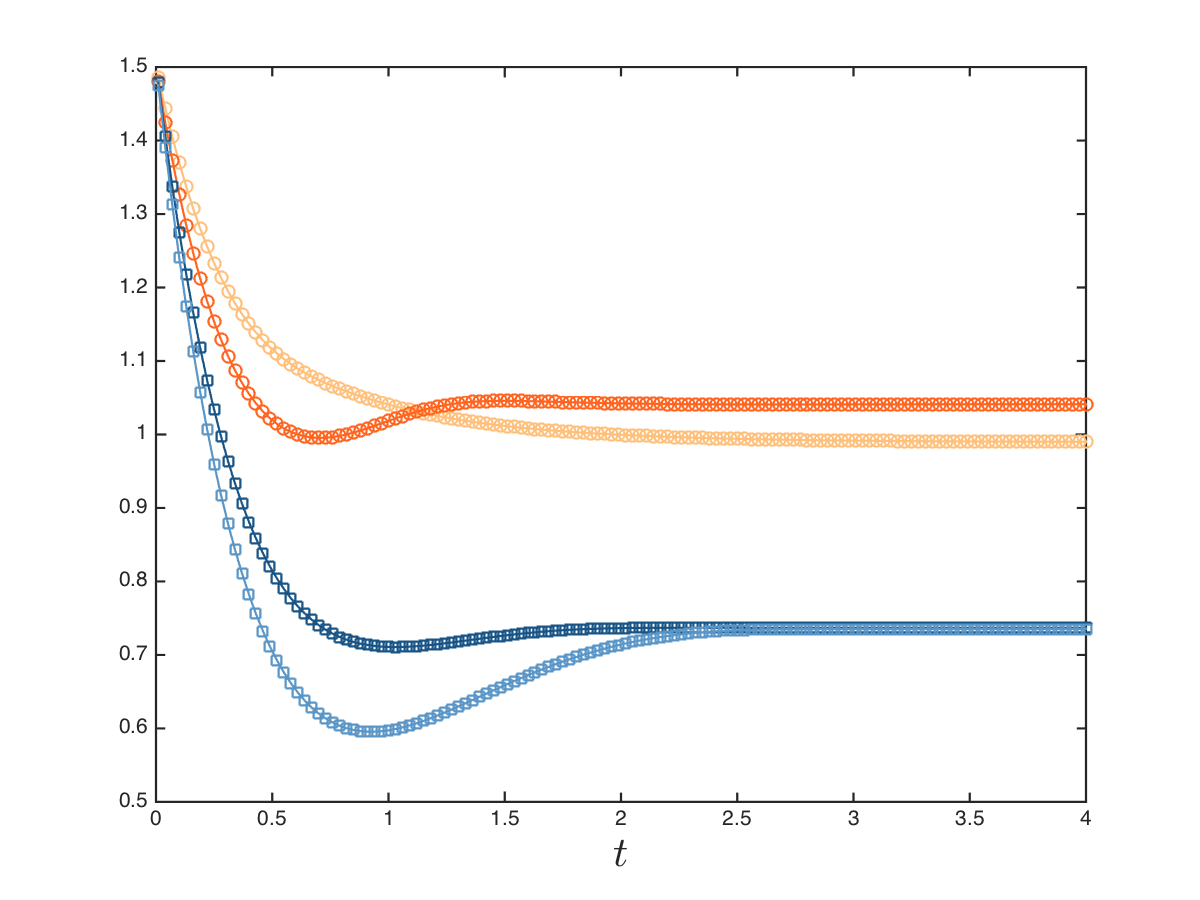}&
\includegraphics[trim=30 10 40 20,clip,width=0.25\textwidth]{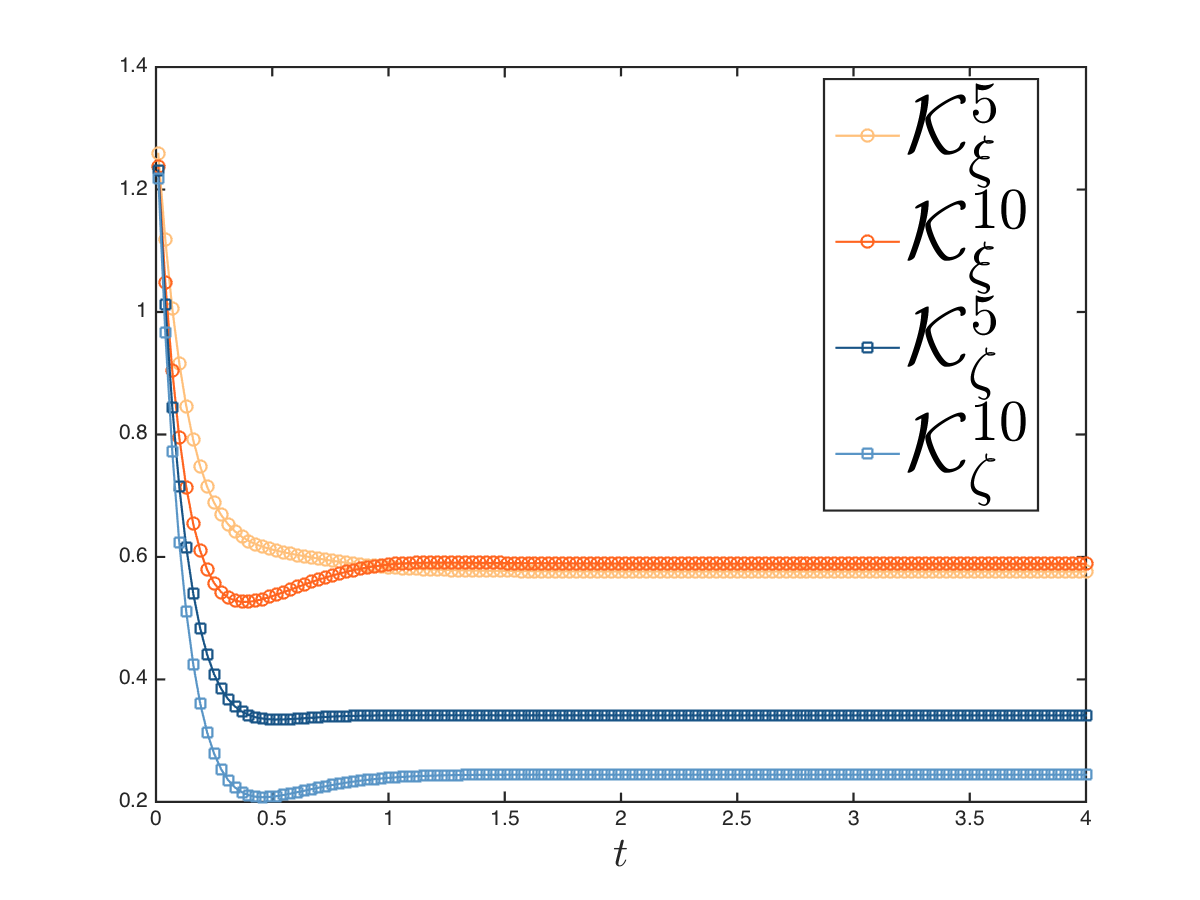}
\\
 \hline
\end{tabular}
\caption{(Test 1a and 1b): Evolution of the stage cost $L[f,\cdot](t)$ at variance of the parameters $R=\{5,10\}$ and $\kappa=\{4,1,0.25\}$ of the tests and for the different controls \eqref{eq:ctrl1} and \eqref{eq:ctrl2}. The stage cost reaches a plateau once the control is inactive and the velocity field reaches an equilibrium.}\label{Fig:F5}
\end{figure}

%

\subsection{Variational stabilization}
In \cite{CFPT12,PRT:15}, the authors propose a sparse control strategy to stabilize a flocking system, targeting the farthest agents from the desired state.  Inspired by these results we define the following {\em selective criteria}, 
\begin{align}\label{eq:varap}
x_\tau^* = \arg\max_{x}\int_{\mathbb{R}^2}\int_{B_{\varrho}(x)}|\bar v-w|^2f(y,w,\tau) \ d(y,w),
\end{align}
where at time $\tau\in[0,T]$ we maximize over the possible center $x\in\overline{\textrm{supp} f}$ of the ball $B_\varrho(x)$, for a given radius $\varrho>0$. Thus the ball $B_\varrho(x_\tau^*)$ represents the selective set for the pointwise control {\eqref{zeta}}, and its characteristic function will be used as the {\em selective function} for the filtered control {\eqref{xi}}.

Note that in general the solution of \eqref{eq:varap}  is computational demanding, to reduce such cost we solve the minimization problem only on a finite sequence of times $\{\tau_\ell\}_{\ell=1}^{L-1}$ of the full time horizon $[0,T]$, and we relay on a Monte-Carlo procedure to approximate the integral in \eqref{eq:varap}.

We report in Figure \ref{fig:7} the comparison of the system density at time $T=4$, having fixed the parameter $\kappa=2$, for both the controls, and for different choices of the radius $\varrho$, $\varrho=\{1.5,2.5,5\}$. The evaluation of \eqref{eq:varap} is done over a  sequence of $L=40$ time intervals, thus for every $t\in[\tau_\ell,\tau_{\ell+1})$ the control remains  defined on on the ball $B_\varrho (x^*_{\tau_\ell})$, we report in each figure the sequence of selective balls, $\{B_\varrho(x^*_{\tau_\ell})\}_{\ell=1}^{L-1}$.  
The first row shows the action of the filtered control \eqref{xi}, whether the second line depicts the action of the pointwise control \eqref{zeta}. 
On one hand in the second row the control seems to perform a better alignment, for a cheaper total cost, $C_T$ in the case of medium values of radius $\varrho$, on the other hand for smaller values of the radius the filtered control seems to perform better. This is more evident in Figure \ref{fig:8} where we plot the running costs $L[f,\cdot](t)$, computed as in \eqref{eq:funcJT} and \eqref{eq:funcJT2}, and where lower cost is obtained respectively, for small value of the radius $\varrho=1.5$ by means of filtered control \eqref{xi}, and for medium-high values of the radius $\varrho=\{2.5,5\}$ in the case of the pointwise control \eqref{zeta}.

\begin{figure}[th]
\centering
\begin{tabular}{@{}c@{\hspace{1mm}}c@{\hspace{1mm}}c@{\hspace{1mm}}c@{}}
\hline
&$\varrho = 1.5$ & $\varrho = 2.5$& $\varrho= 2.5$\\
\hline
\sidecap{$\mathcal{K}_\xi[f](x,v)$} 
&
\includegraphics[trim=30 10 40 20,clip,width=0.25\textwidth]{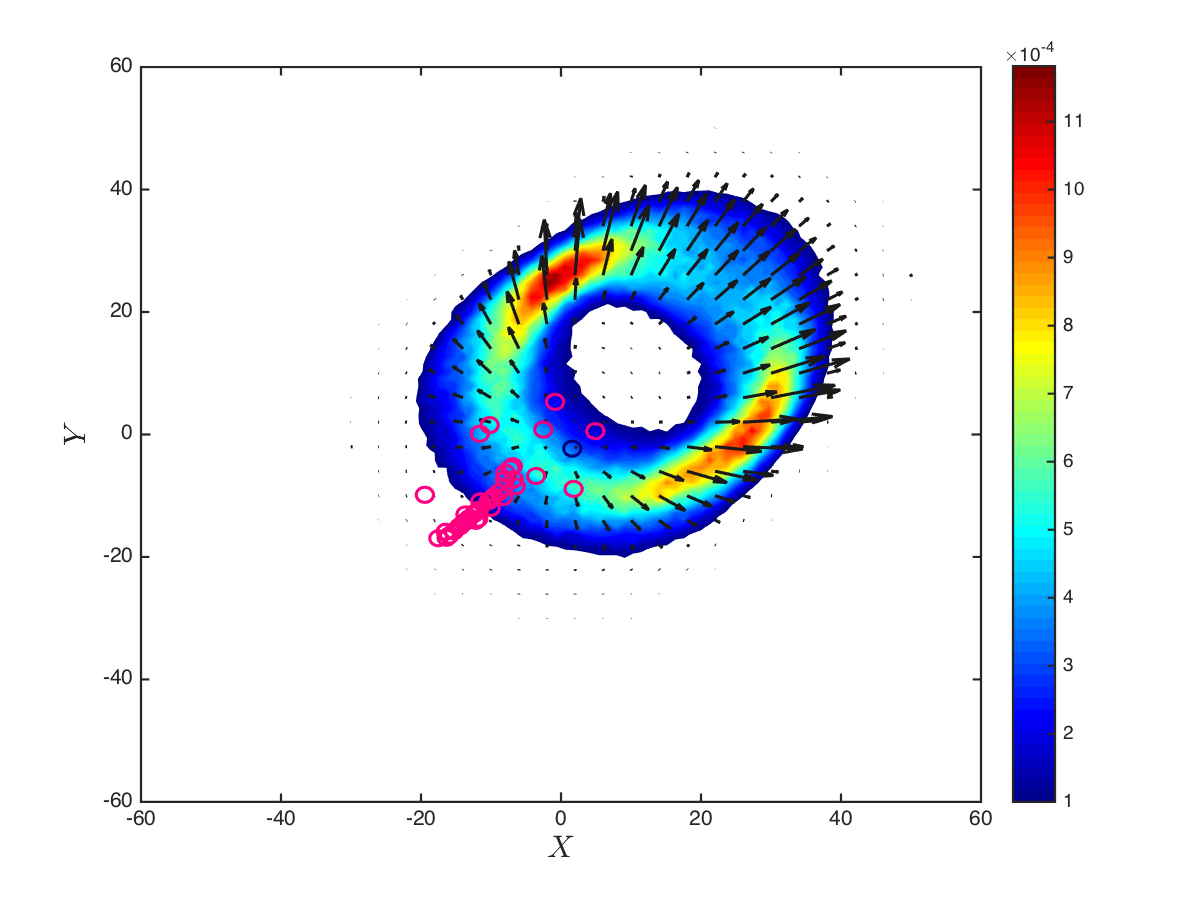}&
\includegraphics[trim=30 10 40 20,clip,width=0.25\textwidth]{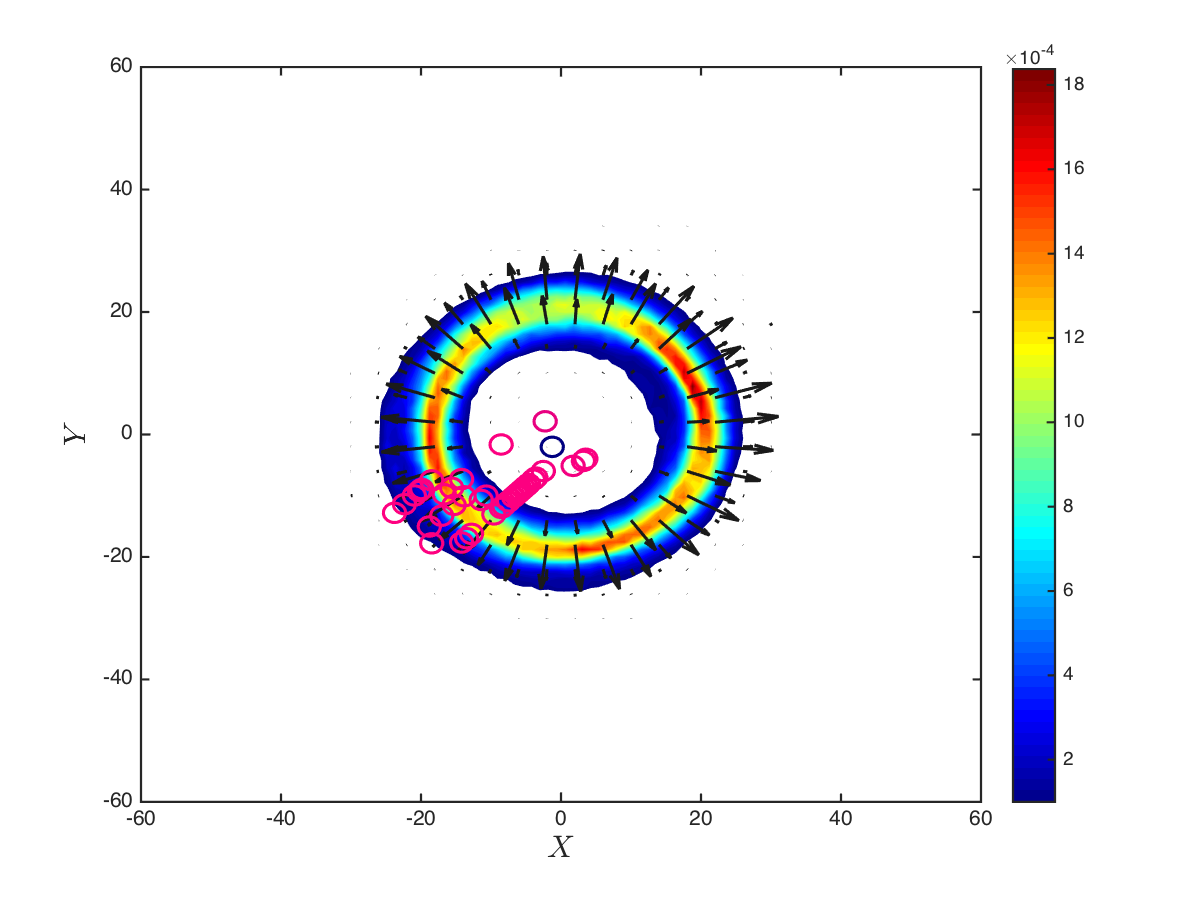}&
\includegraphics[trim=30 10 40 20,clip,width=0.25\textwidth]{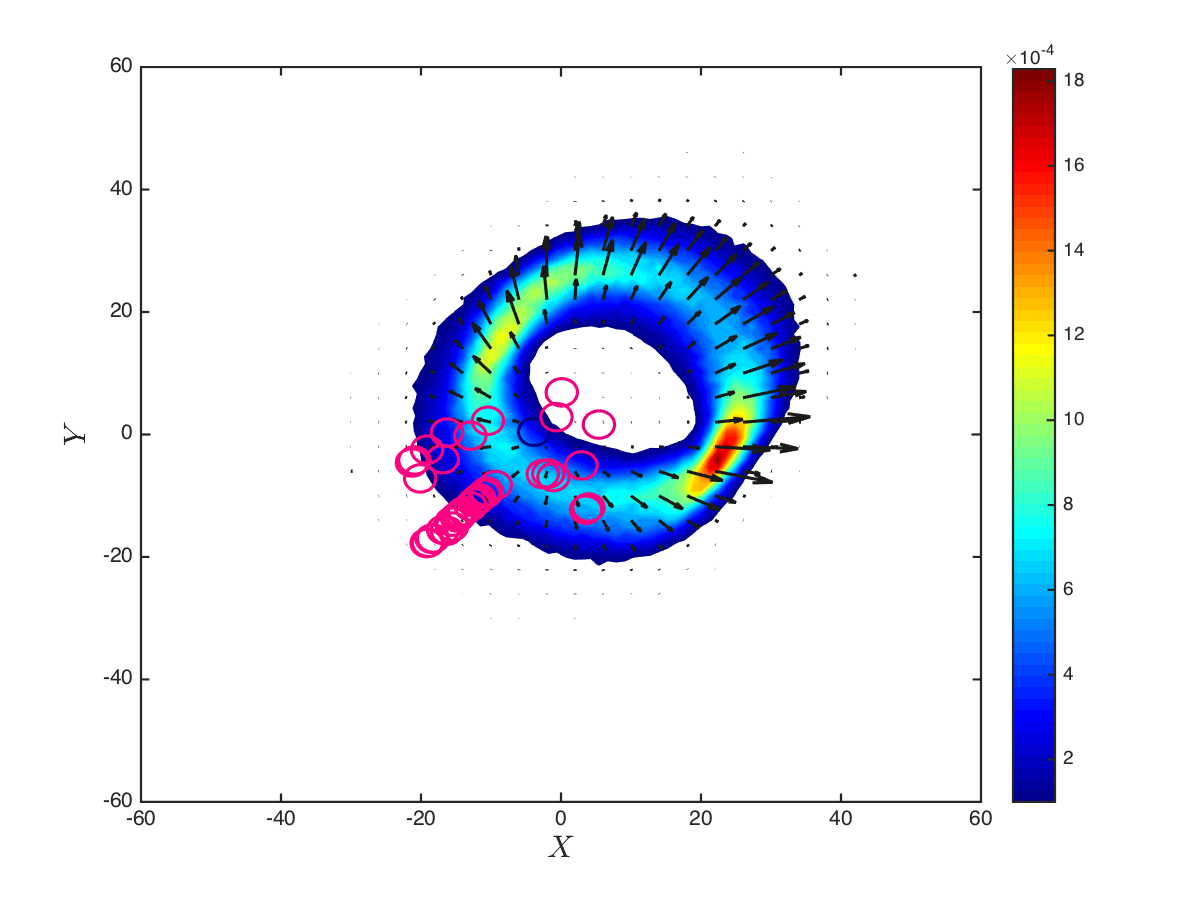}
\\
& $C_T=0.7568 $  & $C_T= 1.2085 $ & $C_T= 0.8223 $ \\
\hline
 \hline
\sidecap{$\mathcal{K}_\zeta[f](x,v)$}
&
\includegraphics[trim=30 10 40 20,clip,width=0.25\textwidth]{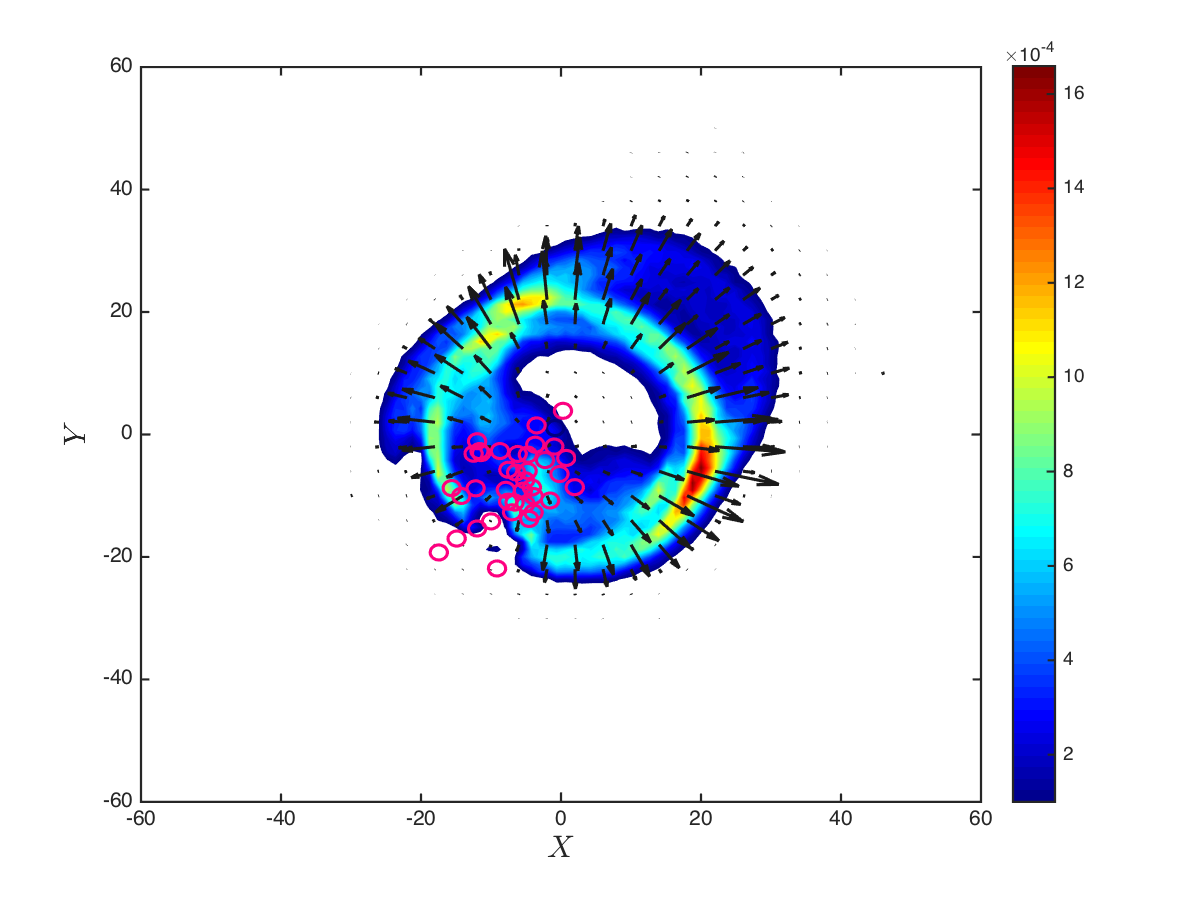}&
\includegraphics[trim=30 10 40 20,clip,width=0.25\textwidth]{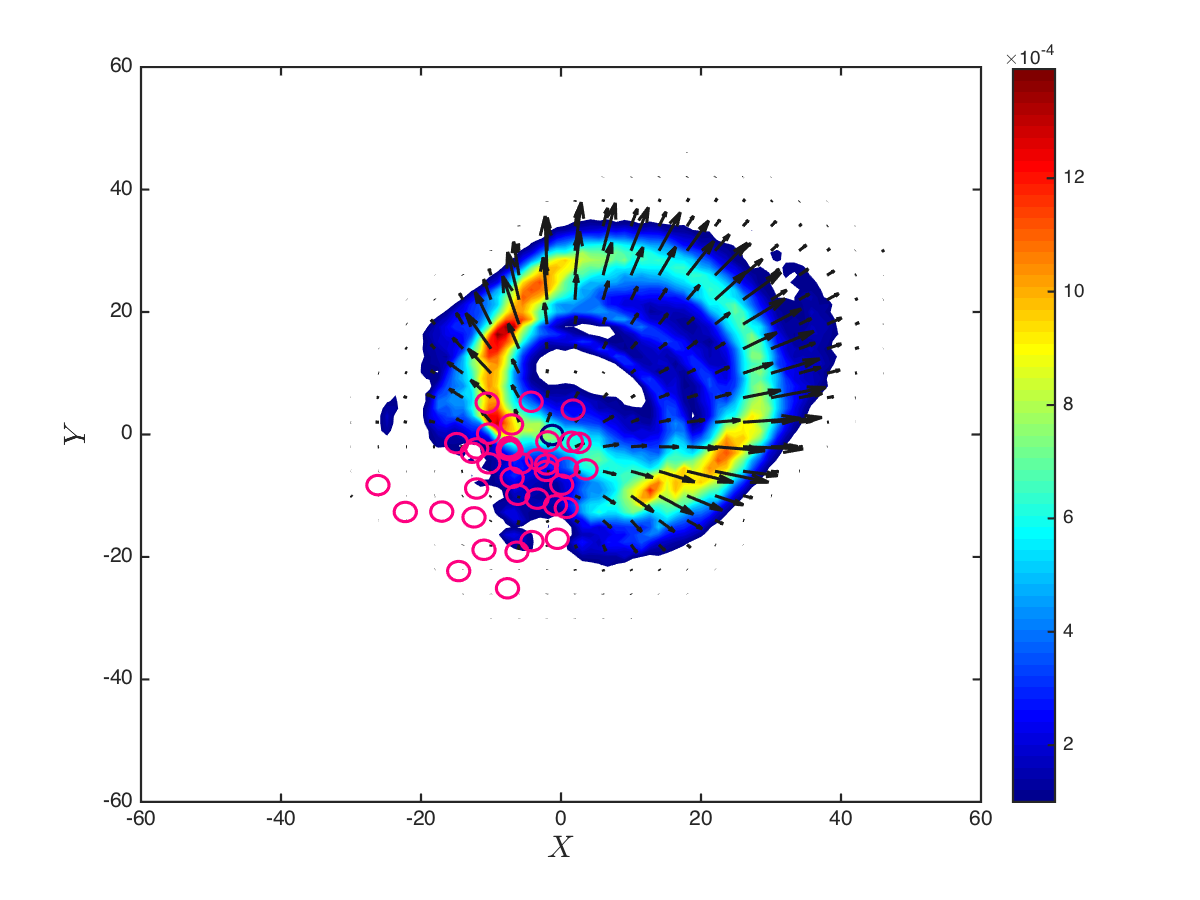}&
\includegraphics[trim=30 10 40 20,clip,width=0.25\textwidth]{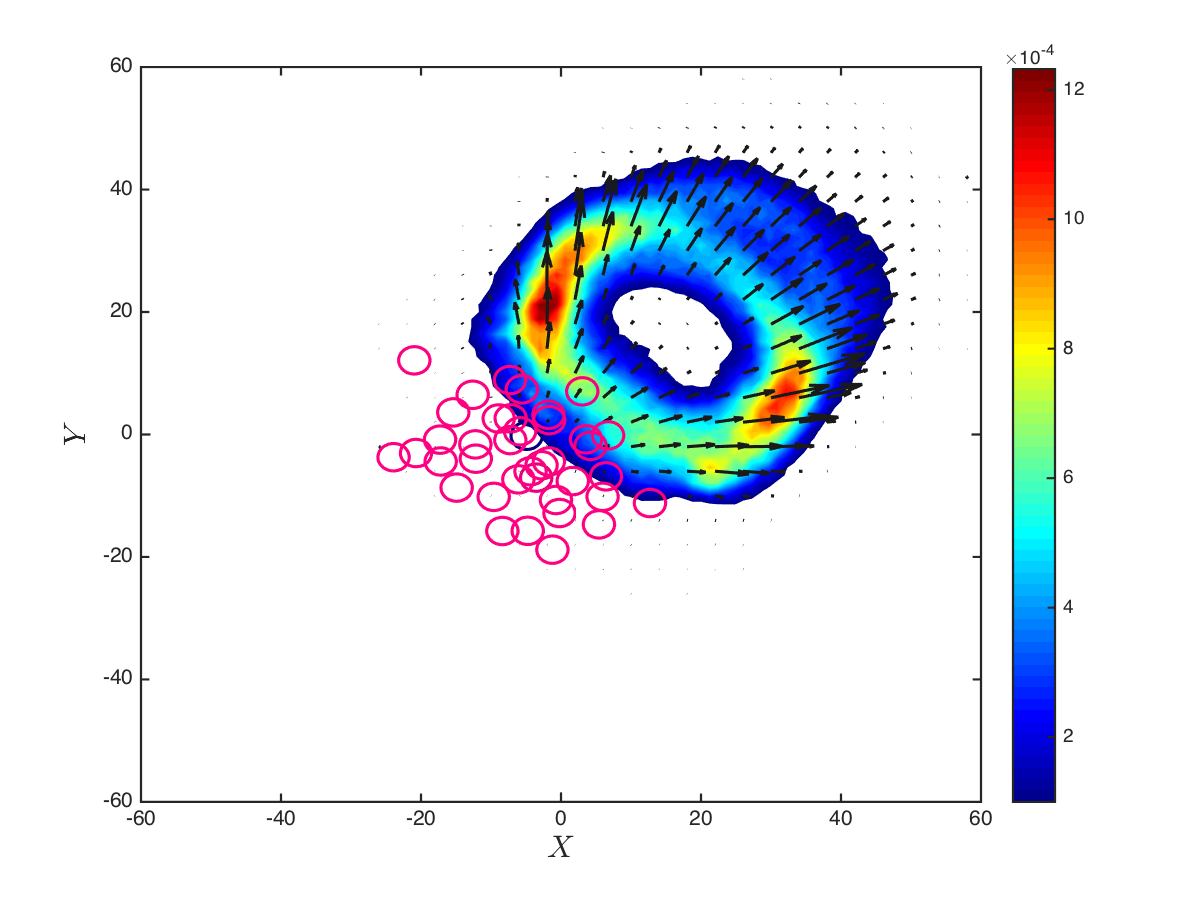}
\\
& $C_T=0.9826  $  & $C_T=0.7169   $ & $C_T= 0.5160 $ \\
 \hline
\end{tabular}
\caption{ Final states at time $T=4$ of the density $f(x,v,T)$ when a the selective control is active only on a region $B_\varrho(x^*_\tau)$ defined in \eqref{eq:varap}, which is represented by the small circle overlapping in the space domain. The top and bottom pictures represent the action of the control, respectively  for $\mathcal{K}_\xi[f](x,v)$, and $\mathcal{K}_\zeta[f](x,v)$, with fixed penalization $\kappa = 0.25$, and for different values of the radius $\varrho$.  The value of the total costs are reported below each picture.}\label{fig:7}
\end{figure}

\begin{figure}[ht]
\centering
\begin{tabular}{@{}c@{\hspace{1mm}}c@{\hspace{1mm}}c@{\hspace{1mm}}c@{}}
\hline
&$\varrho = 1.5$ & $\varrho = 2.5$& $\varrho = 5$\\
\hline
\sidecap{$L[f,\cdot](t)$} 
&
\includegraphics[trim=30 10 40 20,clip,width=0.25\textwidth]{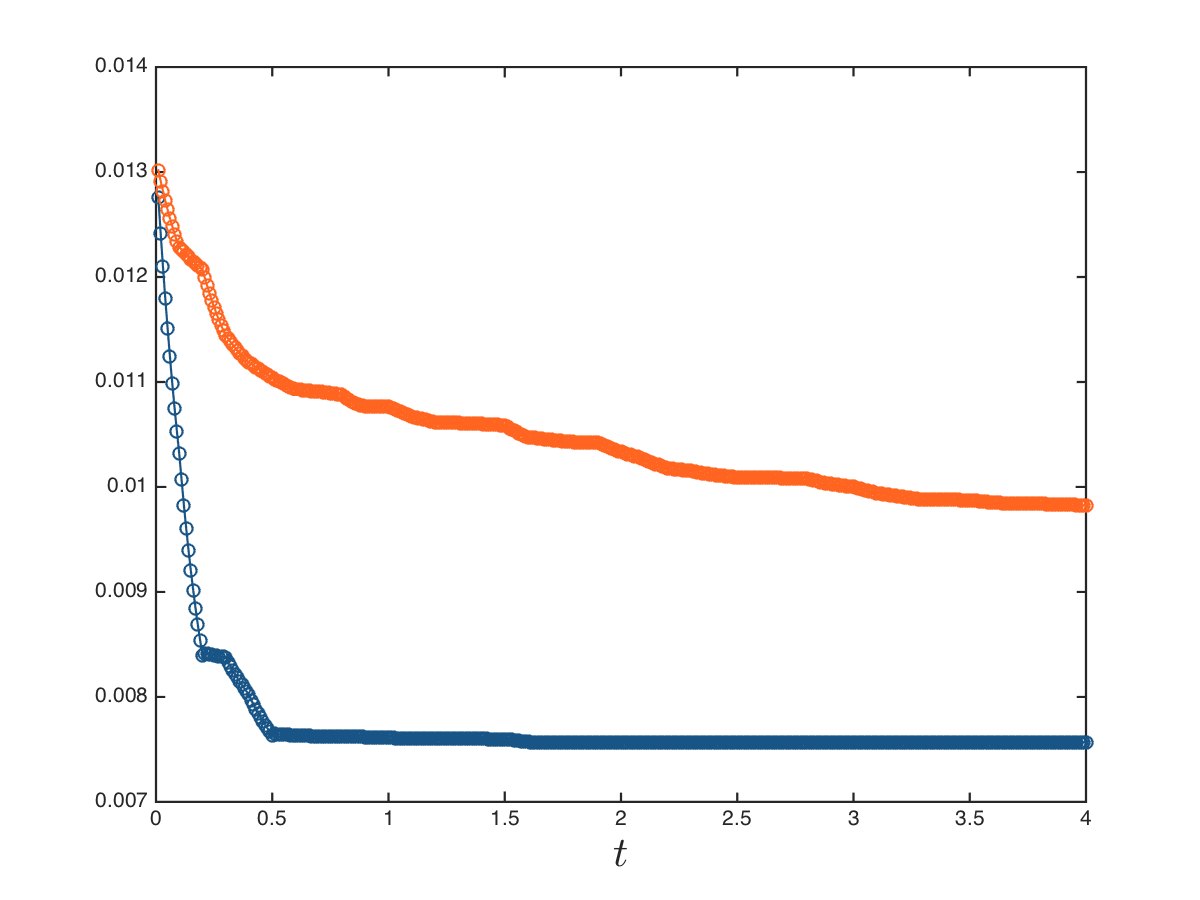}&
\includegraphics[trim=30 10 40 20,clip,width=0.25\textwidth]{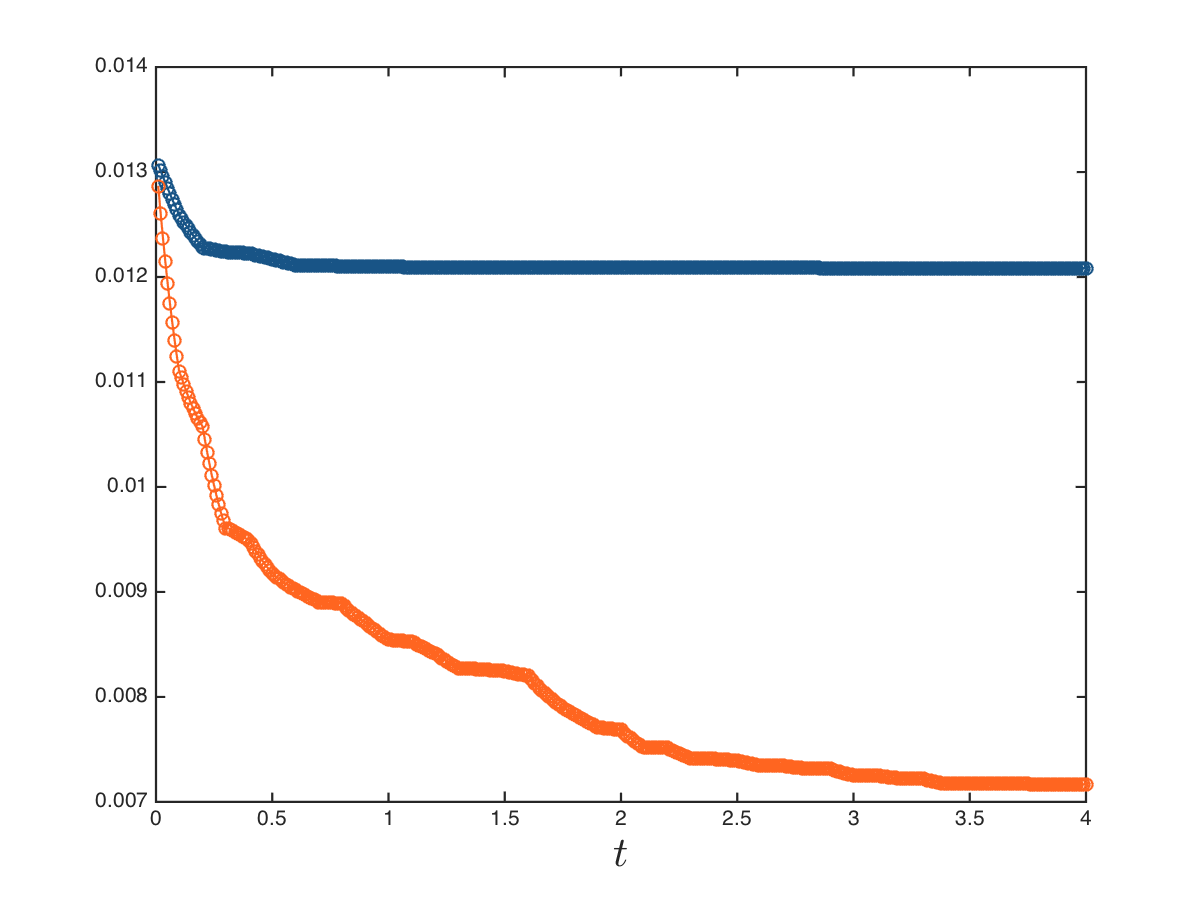}&
\includegraphics[trim=30 10 40 20,clip,width=0.25\textwidth]{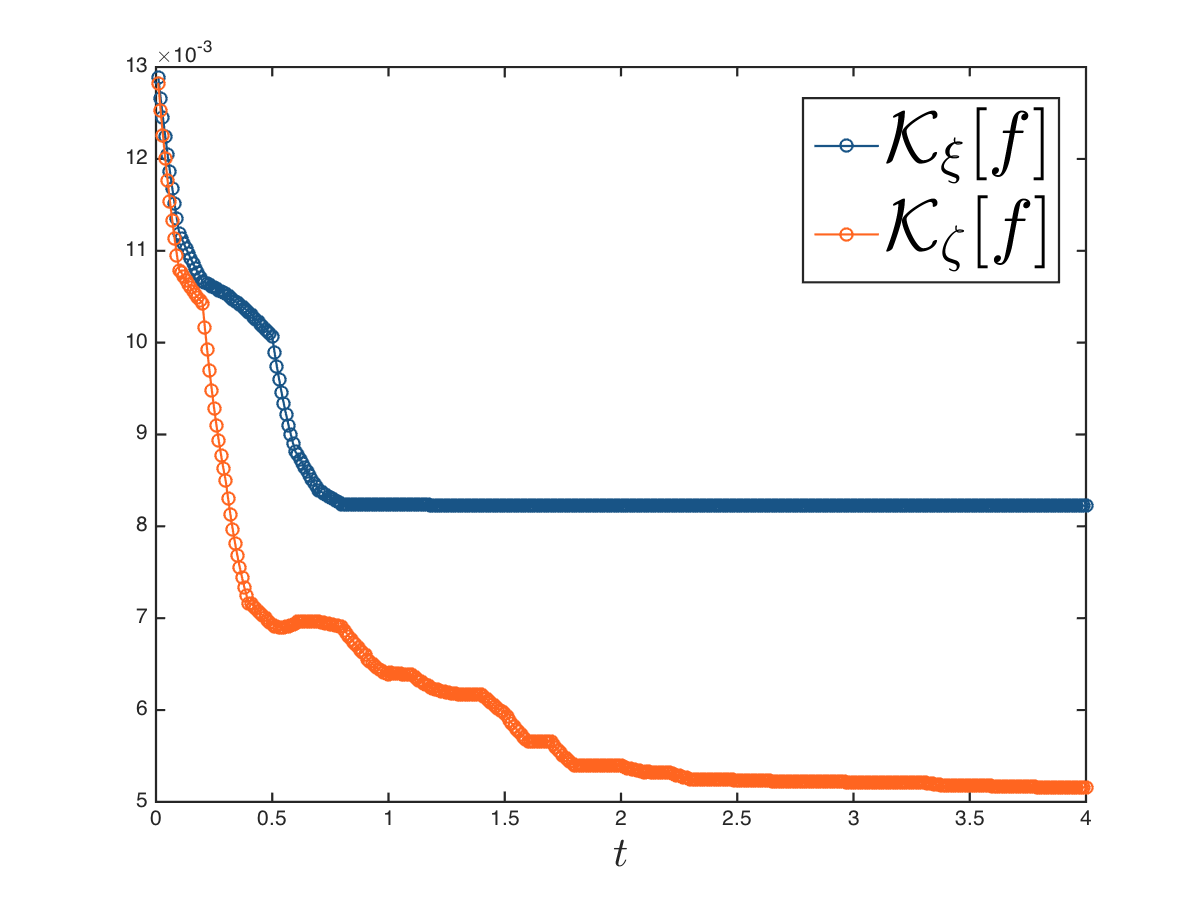}
\\
 \hline
\end{tabular}
\caption{Evolution of the running cost $L[f,\cdot](t)$ at variance of the radius $\varrho=\{1.5,2.5,5\}$, with fixed $\kappa = 0.25$, and with different controls respectively \eqref{eq:ctrl1} and \eqref{eq:ctrl2}.}\label{fig:8}
\end{figure}

%

%

\section*{Conclusions}
We propose a new modeling setting to include a selective action of the control in flocking systems. We showed that a numerical technique procedure, based on model predictive control, furnishes a consistent discretization of a constrained alignment dynamics related to the optimal control problem. The mean-field limit derivation of the controlled particle system has been presented, and proved rigorously in some cases. A fast numerical scheme, based on the stochastic approximation of the interaction operator, has been implemented. Finally several simulations show non trivial behavior of the selective control dynamics, and consolidate the efficacy of the proposed modeling setting. 
\section*{Acknowledgements}
GA acknowledges the ERC-Starting Grant project {\em High-Dimensional
Sparse Optimal Control} (HDSPCONTR).
\appendix
\section{Derivation of the forward feedback control system}\label{app:A}
Let us introduce the parameters $\alpha=\Delta t/N$, and $\beta=\Delta t/\nu$, then we can write system \eqref{eq:Dfwdcontr2} 
as follows
\begin{align}\label{eq:Ffwd}
& v^{n+1}_i=v^n_i+\alpha\sum_{j=1}^{N}H^n_{ij}(v^n_j-v^n_i)-\alpha\beta \sum_{j=1}^N(v_j^{n+1}-\bar v)\Qs^n_j\Qs^n_i,
\end{align}
and in matrix-vector notation as follows
\begin{align}
& \left(\Id + \alpha\beta\QQ^n\right) \vv^{n+1}=\vv^n-\alpha \LL^n\vv^n+\alpha\beta\QQ^n\eee \bar v
\label{eq:VFfwd}
\end{align}
where $\LL^n$ is the laplacian matrix at time $n$, such that $(\LL^n)_{ii} = \sum_{j\neq i}H_{ij}^n$ and $(\LL^n)_{ij} = -H_{ij}^n$, and 
$(\QQ^n)_{ij} = \Qs^n_i\Qs^n_j$,,
and with the following conventional notations, $$\eee=(1,1,\ldots,1)^T,\, \vv^n = (v_1^n,v_2^n,\ldots,v_N^n)^T.$$
We remark that 
 $\A^n:=\left(\Id + \alpha\beta\QQ^n\right)$ enjoys the following properties, namely its inverse is
\begin{align*}
&(\A^n)^{-1}=\Id-\frac{ \alpha\beta}{1+ \alpha\beta\textrm{trac}(\QQ^n)}\QQ^n,
\end{align*}
therefore computing the matrix product we obtain the fully explicit scheme
\begin{equation}
\begin{aligned}
&\vv^{n+1}=\,\vv^n-\alpha\LL^n\vv^n+\frac{ \alpha\beta}{1+ \alpha\beta\textrm{trac}(\QQ^n)}\left(\left(\left(\Id+\alpha\beta\BB^n\right)\QQ^n\eee \bar v-\QQ^n\vv^n\right)+{ \alpha}\QQ^n\LL^n\vv^n\right),
\end{aligned}
\end{equation}
where we defined $\BB^n = \textrm{trac}(\QQ^n)\Id-S^n$.
Expressing again the parameters $\alpha$ and $\beta$ in terms of $\Delta t$, $N$,  and scaling the regularization parameter $\nu = \kappa \Delta t$ we have consequently 
$\beta = 1/\kappa$, which allows to write the previous system as in \eqref{fullDsystem1}.

\section{Binary interaction algorithm}\label{app:B}

We want to approximate the interaction operator in the mean-field model  \eqref{eq:MFmodel} by 
%
 considering pairwise interactions between agents, described by a Boltzmann-like equation under a proper scaling, \cite{AlbiPareschi2013ab, MR2744704}.
Thus we introduce the binary interactions among  two agents $(i,j)$ as follows
\begin{equation}\label{Binary}
\begin{aligned}
&v_i^{*}= v^n_i+\alpha H^n_{ij}(v^n_j-v^n_i)+\alpha K^n_{ij}\\
&v_j^{*}= v^n_j+\alpha H^n_{ji}(v^n_i-v^n_j)+\alpha K^n_{ji}\\
\end{aligned}
\end{equation}
where $(v^*_i,v^*_j)$ are the post interaction velocities, 
$\alpha$ is the interaction strength parameter, and term $K^n_{ij}$ represents the  {\em selective feedback control}, defined as follows,
\begin{equation}
K^n_{ij} =
\begin{cases}
 \displaystyle\frac{2}{2\kappa + \Delta t ((S^n_i)^2+(S^n_j)^2)}\left(\bar v-v^n_j\right)\Qs^n_i\Qs^n_j,
 \\\\
 \displaystyle\frac{1}{\kappa + \Delta t}\left(\bar v-v^n_i\right)B^n_i.
 \end{cases}
\end{equation}

 We denote with  $f^{n}=f(t^n,x,v)$ the non-negative empirical density, described by a sample of $N_s$ particles as follows
\begin{equation}\label{emp}
f^n(x,v) = \frac{1}{N_s}\sum_{i=1}^{N_s}\delta(x-x_i^n)\delta(v-v_i^n).
\end{equation}
Then,  the evolution of the binary interactions is described by the following time discrete version of a Boltzmann-like equation 
\begin{align}
&f^{n+1} =f^n+ \Delta t\lambda\mathcal{Q}^n\label{eq:EqBoltz}\\
&\,\mathcal{Q}^n= \int_{\mathbb{R}^{2d}}\left(\frac{1}{\mathcal{J}}f^n_*f^n_*-f^nf^n\right)\,dw\,dy,\label{eq:IntBoltz}
\end{align}
where $\mathcal{Q}^n=\mathcal{Q}(f^n,f^n)$ is the interaction operator,  $\mathcal{J}(x,v;y,w)$ the Jacobian of the binary transformation \eqref{Binary}, and $f^n_*=f^n(x_*,y_*)$ depicts the {\em pre-interaction}  density such that $(x_*,v_*)\to(x,v)$ after interaction \eqref{Binary}.

In order to see the equivalence of consistency of this numerical scheme, we show that under a {\em grazing interaction} scaling the operator $\mathcal{Q}^n$ is a first order approximation of the non-linear friction operator in \eqref{eq:MFkinetic1}.
 Let us introduce the test function $\varphi(x,v)\in C^2(\mathbb{R}^{2d};\mathbb{R})$, thus we can reformulate \eqref{eq:IntBoltz} in weak form as follows
 \[
\lambda\langle\mathcal{Q}^n,\varphi \rangle= \lambda\int_{\mathbb{R}^{4d}}(\varphi^*-\varphi)f^nf^n\,dwdy\,dvdx,
\]
where $\varphi^*=\varphi(x,v^*)$ and $\varphi=\varphi(x,v)$. Expanding around $v^*-v$, thanks to \eqref{Binary} we have
\begin{equation*}
\begin{aligned}
&\lambda\alpha\int_{\mathbb{R}^{4d}}\nabla_v\varphi\cdot[H(x,y)(w-v)+K(x,v)]f^nf^n\,dwdy\,dvdx+\lambda \alpha^2R[\varphi],
\end{aligned}
\end{equation*}
where $R[\varphi]$ is the bounded remainder of the Taylor expansion.
Introducing the scaling 
$\lambda = 1/\varepsilon$, $\alpha = \varepsilon$, as $\epsi\to 0$ it follows that $\lambda\alpha = 1$ and $\lambda\alpha^2R[\varphi]\to 0$, see \cite{AlbiPareschi2013ab, MR2744704}.  Integrating by parts and reverting back the equation in strong form we finally have
\begin{equation*}
\begin{aligned}
f^{n+1}  = f^{n}-&\Delta t\nabla_v\cdot\left[\int_{\mathbb{R}^{2d}}(H(x,y)(w-v)+K(x,v)f^nf^n\,dydw\right],\\
\end{aligned}
\end{equation*}
such that for $\Delta t \to 0$ it converges to the mean-field equation \eqref{eq:MFkinetic1} in absence of the transport term.

\begin{rmk}
The algorithmic procedure proposed has its roots in the classical Monte Carlo methods of kinetic theory \cite{AlbiPareschi2013ab}.
The full numerical scheme consists in solving iteratively two steps:  transport and interactions,
\begin{align}
&f^{\star} = f^n + \Delta t v\cdot\nabla_x f^n,\tag{T}\\
&f^{n+1} = \left(1-\tau_\varepsilon\right)f^\star + \tau_\varepsilon\mathcal{Q}_\varepsilon^+(f^\star,f^\star)\tag{I}
\end{align}
where $\tau_\varepsilon = \Delta t/\varepsilon$, and $\mathcal{Q}_\varepsilon^+$ represents the gain part of \eqref{eq:IntBoltz}, which accounts the binary exchange of \eqref{Binary}. To preserve the positivity of $f^{n+1}$ the algorithm requires that $\Delta t\leq\varepsilon$, which in general is restrictive since $\epsi$ has to be small. On the other hand, the interaction operator \eqref{eq:IntBoltz} is solved linearly with $O(N_s)$ operations against the $O(N_s^2)$ cost of the direct evaluation. Therefore, the algorithms is efficient whenever $\varepsilon \gg 1/N_s$. For further discussion on the analysis and convergence of the method we refer to \cite{AlbiPareschi2013ab}.
\end{rmk}

\end{document}